\documentclass{article}
\usepackage{fullpage}
\usepackage{amsthm,amssymb,amsfonts,amsmath}
\usepackage{xcolor,hyperref}

\newcommand{\C}{\mathbb{C}}
\newcommand{\F}{\mathbb{F}}
\newcommand{\Q}{\mathbb{Q}}
\newcommand{\Z}{\mathbb{Z}}
\newcommand{\cG}{\mathcal{G}}
\newcommand{\cI}{\mathcal{I}}

\DeclareMathOperator{\Li}{Li}
\DeclareMathOperator{\sqf}{sqf}
\DeclareMathOperator{\ev}{ev}

\numberwithin{equation}{section}

\hypersetup{
    colorlinks,
    linkcolor={red!50!black},
    citecolor={blue!50!black},
    urlcolor={blue!80!black}
}

\newtheorem{theorem}{Theorem}[section]
\newtheorem{lemma}[theorem]{Lemma}
\newtheorem{proposition}[theorem]{Proposition}
\newtheorem{corollary}[theorem]{Corollary}
\newtheorem{definition}{Definition}[section]

\allowdisplaybreaks

\title{Density of Special Classes of Polynomials with Squarefree Discriminant}
\author{Gian Cordana Sanjaya}

\begin{document}

\maketitle

\begin{abstract}
In this paper, we consider the problem of determining the density of monic polynomials over $\Z_p$ with squarefree discriminant over various subsets of the set of monic polynomials over $\Z_p$ of fixed degree.
We compute the density of polynomials in each subset whose discriminant is squarefree, and we compute the density of polynomials $f$ in each subset such that $\Z_p[x]/(f(x))$ is the maximal order of $\Q_p[x]/(f(x))$.
\end{abstract}

\section{Introduction}

It is a classical question in analytic number theory to determine, given an integer polynomial $F$ over $n$ variables, the density of integer points $\mathbf{a}$ such that $F(\mathbf{a})$ is squarefree, where the points $\mathbf{a}$ are ordered by some height function.
In the classical case $F(x) = x$, this density is equal to $\zeta(2)^{-1} = 6/\pi^2$.
In general, one conjectures that the density to be computed is equal to the Euler product, over all prime numbers $p$, of the density of $p$-adic points $\mathbf{a}$ such that $p^2$ does not divide $F(\mathbf{a})$, so-called the local density at $p$.

Recently, some work has been done in the case where $F$ is the discriminant of a monic integer polynomial $f$ of degree $n$, viewed as a polynomial in terms of the coefficients of $f$.
In~\cite{BSW2022}, Bhargava, Shankar, and Wang computed the density of monic integer polynomials of degree $n$ with squarefree discriminant.
They also determined the density of monic integer polynomials $f$ of degree $n$ such that $\Z[x]/(f(x))$ is the maximal order of $\Q[x]/(f(x))$.
Previously, the local density of monic polynomials over $\Z_p$ of degree $n$ with squarefree discriminant was computed by Yamamura~\cite{Yam1991}, and the local density of monic polynomials $f \in \Z_p[x]$ of degree $n$ such that $\Z_p[x]/(f(x))$ is the maximal order of $\Q_p[x]/(f(x))$ was computed by Lenstra~\cite{ABZ2007}.

In this paper, we compute the density of monic polynomials $f \in \Z_p[x]$ of degree $n$ with squarefree discriminant, assuming some restrictions on the coefficients of $f$.
We also compute the density of monic polynomials $f \in \Z_p[x]$ of degree $n$ such that $\Z_p[x]/(f(x))$ is the maximal order of $\Q_p[x]/(f(x))$, also assuming some restrictions on the coefficients of $f$.
Our method allows us not only to compute the asymptotic densities as $n \to \infty$, but also to compute the \emph{exact} densities.

Let $V_n(\Z_p)$ be the set of monic polynomials over $\Z_p$ of degree $n$.
For any subset $\Sigma \subseteq V_n(\Z_p)$, we denote by $P^{\sqf}(\Sigma)$ the density of $f \in \Sigma$ such that $f$ has squarefree discriminant, and we denote by $P^{\max}(\Sigma)$ the density of $f \in \Sigma$ such that $\Z_p[x]/(f(x))$ is the maximal order of $\Q_p[x]/(f(x))$.
We view $\Sigma$ as a subset of $\Z_p^{n - k}$ (with positive measure) if $\Sigma$ is defined by exactly $k$ coefficients fixed and the other $n - k$ coefficients defined by congruence conditions.

We first consider the set of polynomials in $V_n(\Z_p)$ whose $x^{n - 1}$-coefficient is fixed.
We prove:

\begin{theorem}\label{thm-local:sub1}
Let $n \geq 2$ and $p$ be a prime number.
Fix $b_1 \in \Z_p$, and denote
\[ V_n(\Z_p)_{a_1 = b_1} = \{f(x) = x^n + a_1 x^{n - 1} + \ldots + a_n \in \Z_p[x] : a_1 = b_1\}. \]
If $p = 2$, then
\[ P^{\sqf}(V_n(\Z_2)_{a_1 = b_1}) = \begin{cases} 1 & \text{if } n = 2 \text{ and } 2 \nmid b_1, \\ 0 & \text{if } n = 2 \text{ and } 2 \mid b_1, \\ 1/2 & \text{if } n \geq 3, \end{cases} \qquad P^{\max}(V_n(\Z_2)_{a_1 = b_1}) = \begin{cases} 1 & \text{if } n = 2 \text{ and } 2 \nmid b_1, \\ 1/2 & \text{if } n = 2 \text{ and } 2 \mid b_1, \\ 3/4 & \text{if } n \geq 3. \end{cases} \]
If $p$ is odd, then
\[ P^{\sqf}(V_n(\Z_p)_{a_1 = b_1}) = 1 - \frac{3p - 1}{p^2 (p + 1)} + O\left(\frac{\gcd(n, p)}{p^{n - 1}}\right), \qquad 
    P^{\max}(V_n(\Z_p)_{a_1 = b_1}) = 1 - \frac{1}{p^2}, \]
    where the implied constants are absolute.
If $p$ is odd, then $P^{\sqf}(V_n(\Z_p)_{a_1 = b_1})$ is positive, independent of $b_1$ if $p \nmid n$, and depends strictly on whether $b_1$ is a unit or not a unit if $p \mid n$.
\end{theorem}

In particular, the asymptotic densities of $V_n(\Z_p)_{a_1 = b_1}$ as $n \to \infty$ are equal to those of $V_n(\Z_p)$ as obtained by Yamamura~\cite[Proposition 3]{Yam1991} (for squarefree discriminant) and by Lenstra~\cite[Proposition 3.5]{ABZ2007} (for maximality).
It is not surprising that the densities of $V_n(\Z_p)_{a_1 = b_1}$ is independent of $b_1$ if $p \nmid n$; the point of interest is in the case $p \mid n$.
Note that for $n \geq 3$, we have $\gcd(n, p)/p^{n - 1} \ll_n 1/p^2$, which implies that for any integer $b_1$,
\[ \prod_p P^{\sqf}(V_n(\Z_p)_{a_1 = b_1}) > 0. \]

Next, we consider the set of polynomials in $V_n(\Z_p)$ whose $x^{n - 1}$- and $x^{n - 2}$-coefficients are fixed.

\begin{theorem}\label{thm-local:sub2}
Let $n \geq 3$ and $p$ be a prime number.
Fix $b_1, b_2 \in \Z_p$, and denote
\[ V_n(\Z_p)_{a_1 = b_1, a_2 = b_2} = \{f(x) = x^n + a_1 x^{n - 1} + \ldots + a_n \in \Z_p[x] : a_1 = b_1, a_2 = b_2\}. \]
If $p$ is odd, then
\begin{align*}
    P^{\sqf}(V_n(\Z_p)_{a_1 = b_1, a_2 = b_2}) &= 1 - \frac{3p - 1}{p^2 (p + 1)} + O\left(\frac{\gcd(n, p)}{p^{\lfloor 3n/4 \rfloor - 1}}\right), & P^{\sqf}(V_n(\Z_2)_{a_1 = b_1, a_2 = b_2}) &= \frac{1}{2}, \\
    P^{\max}(V_n(\Z_p)_{a_1 = b_1, a_2 = b_2}) &= 1 - \frac{1}{p^2} + O\left(\frac{\gcd(n, p)}{p^{\lfloor 3n/4 \rfloor + \lfloor n/2 \rfloor - 1}}\right), & P^{\max}(V_n(\Z_2)_{a_1 = b_1, a_2 = b_2}) &= \frac{3}{4},
\end{align*}
    where the implied constants are absolute.
These densities are positive if $n \geq 4$.
If $b_1$ and $b_2$ are integers, then $P^{\max}(V_n(\Z_p)_{a_1 = b_1, a_2 = b_2}) = 1 - 1/p^2$ holds for all $p$ if and only if $n$ is a power of $2$ and $2 b_2 - (1 - 1/n) b_1^2 = 0$.
\end{theorem}

The quantity $2 b_2 - (1 - 1/n) b_1^2$ has a natural interpretation as an invariant of $V_n(\Z_p)_{a_1 = b_1, a_2 = b_2}$;
    there is a bijection of form $f(x) \mapsto f(x + c)$ between $V_n(\Z_p)_{a_1 = b_1, a_2 = b_2}$ and $V_n(\Z_p)_{a_1 = c_1, a_2 = c_2}$ if and only if $2 b_2 - (1 - 1/n) b_1^2 = 2 c_2 - (1 - 1/n) c_1^2$.
Note that the case $n = 4$, $(b_1, b_2) = (0, 0)$ has been considered by Sanjaya and Wang in~\cite{SW2023}.

Now we consider the set of polynomials in $V_n(\Z_p)$ whose constant coefficient is a unit.
We prove:

\begin{theorem}\label{thm-local:copconst}
Let $n \geq 2$ and $p$ be a prime number.
Denote
\[ V_n(\Z_p)_{p \nmid a_n} = \{f(x) = x^n + a_1 x^{n - 1} + \ldots + a_n \in \Z_p[x] : p \nmid a_n\}. \]
Then
\begin{align*}
    P^{\sqf}(V_n(\Z_p)_{p \nmid a_n}) &= \begin{cases}
        \dfrac{2}{3} \left(1 - \dfrac{(-1)^n}{2^n}\right) & \text{if } p = 2, \\
        1 - \dfrac{3p - 1}{p(p + 1)^2} + O\left(\dfrac{n}{p^{n - 1}}\right) & \text{if } p \neq 2, \end{cases} \\
    P^{\max}(V_n(\Z_p)_{p \nmid a_n}) &= 1 - \frac{1}{p^2 + p + 1} + O\left(\frac{1}{p^{n + \lceil n/2 \rceil}}\right),
\end{align*}
    where the implied constants are absolute, and these densities are positive.
\end{theorem}

Note the difference of the densities of $V_n(\Z_p)_{p \nmid a_n}$ and $V_n(\Z_p)$ in the limit as $n \to \infty$.
The main reason for the difference is that modulo $p$, all elements of $V_n(\Z_p)_{p \nmid a_n}$ are not divisible by $x$, while such divisibility restrictions do not exist for $V_n(\Z_p)$.

Next, we consider the set of polynomials in $V_n(\Z_p)$ whose constant coefficient is a fixed $p$-adic unit.

\begin{theorem}\label{thm-local:fixconst}
Let $n \geq 2$ and $p$ be a prime number.
Fix some $b_n \in \Z_p^{\times}$, and denote
\[ V_n(\Z_p)_{a_n = b_n} = \{f(x) = x^n + a_1 x^{n - 1} + \ldots + a_n \in \Z_p[x] : a_n = b_n\}. \]
Then
\begin{align*}
    P^{\sqf}(V_n(\Z_p)_{a_n = b_n}) &= \begin{cases} \dfrac{2}{3} \left(1 - \dfrac{(-1)^n}{2^n}\right) & \text{if } p = 2, \\ 1 - \dfrac{3p - 1}{p(p + 1)^2} + O\left(\dfrac{n}{p^{2 \lceil n/2 \rceil - 2}}\right) & \text{if } p \neq 2,
        \end{cases} \\
    P^{\max}(V_n(\Z_p)_{a_n = b_n}) &= 1 - \frac{1}{p^2 + p + 1} + O\left(\frac{1}{p^{3 \lceil n/2 \rceil - 1}}\right),
\end{align*}
    where the implied constants are absolute, and these densities are positive.
The quantity $P^{\sqf}(V_n(\Z_p)_{a_n = b_n})$ is independent of $b_n$ if and only if $\gcd(n, p - 1) = 1$.
The quantity $P^{\max}(V_n(\Z_p)_{a_n = b_n})$ depends strictly on whether $b_n$ is a square or not a square mod $p$ if $p$ is odd and $n$ is even, and is independent of $b_n$ otherwise.
\end{theorem}

Now, we consider the set of polynomials in $V_n(\Z_p)$ whose $x^{n - 1}$-coefficient is fixed and whose constant coefficient is a unit.
We prove:

\begin{theorem}\label{thm-local:sub1-copconst}
Let $n \geq 2$ and $p$ be a prime number.
Fix some $b_1 \in \Z_p$, and denote
\[ V_n(\Z_p)_{a_1 = b_1, p \nmid a_n} = \{f(x) = x^n + a_1 x^{n - 1} + \ldots + a_n \in \Z_p[x] : a_1 = b_1, p \nmid a_n\}. \]
Then
\begin{align*}
    P^{\sqf}(V_n(\Z_p)_{a_1 = b_1, p \nmid a_n}) &= \begin{cases} 2/3 + O(1/2^n) & \text{if } p = 2, \\ 1 - \dfrac{3p - 1}{p(p + 1)^2} + O\left(\dfrac{n}{p^{n - 1}}\right) & \text{if } p \neq 2, \end{cases} \\
    P^{\max}(V_n(\Z_p)_{a_1 = b_1, p \nmid a_n}) &= 1 - \frac{1}{p^2 + p + 1} + O\left(\frac{1}{p^{n + \lfloor n/2 \rfloor - 1}}\right),
\end{align*}
    where the implied constants are absolute.
These densities are positive, with one exception: if $b_1 \in \Z_2$ is not a unit, then $P^{\sqf}(V_2(\Z_2)_{a_1 = b_1, 2 \nmid a_n}) = 0$.
Both densities depend strictly on whether $b_1$ is a unit in $\Z_p$ or not.
\end{theorem}

Next, we consider the following subsets of $V_n(\Z_p)$: (1). those whose constant and $x^{n - 1}$-coefficients are units; and (2). those whose constant and $x$-coefficients are units.
The densities of the two sets are easily seen to be equal by the bijection defined by dividing the polynomial with the constant coefficient and then reflecting the polynomial.
The density of the first set is equal to the average of the corresponding densities of $V_n(\Z_p)_{a_1 = b_1, p \nmid a_n}$ over all $b_1 \in \Z_p^{\times}$, which is independent of $b_1$ as seen above.
Thus, the following result follows immediately from Theorem~\ref{thm-local:sub1-copconst}.

\begin{theorem}\label{thm-local:copconstplus}
Let $n \geq 2$ and $p$ be a prime number.
Denote
\begin{align*}
    V_n(\Z_p)_{p \nmid a_1, p \nmid a_n} &= \{f(x) = x^n + a_1 x^{n - 1} + \ldots + a_n \in \Z_p[x] : p \nmid a_1, p \nmid a_n\}, \\
    V_n(\Z_p)_{p \nmid a_{n - 1}, p \nmid a_n} &= \{f(x) = x^n + a_1 x^{n - 1} + \ldots + a_n \in \Z_p[x] : p \nmid a_{n - 1}, p \nmid a_n\}.
\end{align*}
Define $V_n(\Z_p)_{a_1 = 1, p \nmid a_n}$ as in Theorem~\ref{thm-local:sub1-copconst}.
Then
\begin{align*}
    P^{\sqf}(V_n(\Z_p)_{p \nmid a_1, p \nmid a_n}) &= P^{\sqf}(V_n(\Z_p)_{a_1 = 1, p \nmid a_n}), & P^{\sqf}(V_n(\Z_p)_{p \nmid a_{n - 1}, p \nmid a_n}) &= P^{\sqf}(V_n(\Z_p)_{a_1 = 1, p \nmid a_n}), \\
    P^{\max}(V_n(\Z_p)_{p \nmid a_1, p \nmid a_n}) &= P^{\max}(V_n(\Z_p)_{a_1 = 1, p \nmid a_n}), & P^{\max}(V_n(\Z_p)_{p \nmid a_{n - 1}, p \nmid a_n}) &= P^{\max}(V_n(\Z_p)_{a_1 = 1, p \nmid a_n}),
\end{align*}
    and these densities are positive.
\end{theorem}

In~\cite{ISW2025}, Iverson, Sanjaya, and Wang computed the density of monic integer polynomials $f$ of degree $n$ with prime non-leading coefficients such that $f$ has squarefree discriminant.
They also computed the density of monic integer polynomials $f$ of degree $n$ with prime non-leading coefficients such that $\Z[x]/(f(x))$ is the maximal order of $\Q[x]/(f(x))$.
Here, we consider a variant of the problem, where we set only a few non-leading coefficients, as opposed to all of them, to be prime.

For any $\cI \subseteq \{1, 2, \ldots, n\}$ and real number $X > 0$, denote
\begin{align*}
    V_{n; \cI}(X) &= \{f(x) = x^n + a_1 x^{n - 1} + \ldots + a_n \in \Z[x] : |a_i| < X^i \text{ for all } i \leq n \text{ and } a_i \text{ is prime for all } i \in \cI\}, \\
    N_{n; \cI}^{\sqf}(X) &= \#\{f \in V_{n; \cI}(X) : f \text{ has squarefree discriminant}\}, \\
    N_{n; \cI}^{\max}(X) &= \#\{f \in V_{n; \cI}(X) : \Z[x]/(f(x)) \text{ is the maximal order of } \Q[x]/(f(x))\}.
\end{align*}
First, we consider the case where only the constant coefficient is restricted to be prime.

\begin{theorem}\label{thm-global:const}
For any $n \geq 2$ and for any real number $A > 0$, we have
\begin{align*}
    N_{n; \{n\}}^{\sqf}(X) &= C_{n; \{n\}}^{\sqf} \Li(X^n) \prod_{i = 1}^{n - 1} (2X^i) + O_A\left(\frac{X^{n(n + 1)/2}}{(\log X)^A}\right), \\
    N_{n; \{n\}}^{\max}(X) &= C_{n; \{n\}}^{\max} \Li(X^n) \prod_{i = 1}^{n - 1} (2X^i) + O_A\left(\frac{X^{n(n + 1)/2}}{(\log X)^A}\right),
\end{align*}
    with \[ C_{n; \{n\}}^{\sqf} = \prod_p P^{\sqf}(V_n(\Z_p)_{p \nmid a_n}) > 0, \qquad
        C_{n; \{n\}}^{\max} = \prod_p P^{\max}(V_n(\Z_p)_{p \nmid a_n}) > 0, \]
    where $P^{\sqf}(V_n(\Z_p)_{p \nmid a_n})$ and $P^{\max}(V_n(\Z_p)_{p \nmid a_n})$ are as defined in Theorem~\ref{thm-local:copconst}.
\end{theorem}

Second, we consider the case where the constant and $x^{n - 1}$-coefficients are restricted to be prime.

\begin{theorem}\label{thm-global:sub1-const}
For any $n \geq 2$ and for any real number $A > 0$, we have
\begin{align*}
    N^{\sqf}_{n; \{1, n\}}(X) &= C_{n; \{1, n\}}^{\sqf} \Li(X) \Li(X^n) \prod_{i = 2}^{n - 1} (2X^i) + O_A\left(\frac{X^{n(n + 1)/2}}{(\log X)^A}\right), \\
    N^{\max}_{n; \{1, n\}}(X) &= C_{n; \{1, n\}}^{\max} \Li(X) \Li(X^n) \prod_{i = 2}^{n - 1} (2X^i) + O_A\left(\frac{X^{n(n + 1)/2}}{(\log X)^A}\right),
\end{align*}
    with \[ C_{n; \{1, n\}}^{\sqf} = \prod_p P^{\sqf}(V_n(\Z_p)_{p \nmid a_1, p \nmid a_n}) > 0, \qquad
        C_{n; \{1, n\}}^{\max} = \prod_p P^{\max}(V_n(\Z_p)_{p \nmid a_1, p \nmid a_n}) > 0, \]
    where $P^{\sqf}(V_n(\Z_p)_{p \nmid a_1, p \nmid a_n})$ and $P^{\max}(V_n(\Z_p)_{p \nmid a_1, p \nmid a_n})$ are as defined in Theorem~\ref{thm-local:copconstplus}.
\end{theorem}

Third, we consider the case where the constant and $x$-coefficients are restricted to be prime.

\begin{theorem}\label{thm-global:x-const}
For any $n \geq 2$ and for any real number $A > 0$, we have
\begin{align*}
    N^{\sqf}_{n; \{n - 1, n\}}(X) &= C_{n; \{n - 1, n\}}^{\sqf} \Li(X^{n - 1}) \Li(X^n) \prod_{i = 1}^{n - 2} (2X^i) + O_A\left(\frac{X^{n(n + 1)/2}}{(\log X)^A}\right), \\
    N^{\max}_{n; \{n - 1, n\}}(X) &= C_{n; \{n - 1, n\}}^{\max} \Li(X^{n - 1}) \Li(X^n) \prod_{i = 1}^{n - 2} (2X^i) + O_A\left(\frac{X^{n(n + 1)/2}}{(\log X)^A}\right),
\end{align*}
    with \[ C_{n; \{n - 1, n\}}^{\sqf} = \prod_p P^{\sqf}(V_n(\Z_p)_{p \nmid a_{n - 1}, p \nmid a_n}) > 0, \qquad
        C_{n; \{n - 1, n\}}^{\max} = \prod_p P^{\max}(V_n(\Z_p)_{p \nmid a_{n - 1}, p \nmid a_n}) > 0, \]
    where $P^{\sqf}(V_n(\Z_p)_{p \nmid a_{n - 1}, p \nmid a_n})$ and $P^{\max}(V_n(\Z_p)_{p \nmid a_{n - 1}, p \nmid a_n})$ are as defined in Theorem~\ref{thm-local:copconstplus}.
\end{theorem}

We now sketch the proofs of our results.
We cannot apply~\cite[Theorem 1.3]{ISW2025} directly to obtain Theorem~\ref{thm-global:const}-\ref{thm-global:x-const}, because it only counts $n$-tuples of integers in a set with all entries being prime.
However,~\cite[Theorem 1.3]{ISW2025} can be generalized to count $n$-tuples of integers with entries restricted to be prime only at a fixed set of indices.
The proof follows by modifying the counting argument in the small range, where we apply the Siegel-Walfisz theorem only on the entries restricted to be prime and use naive estimate on the other entries.
This generalization of~\cite[Theorem 1.3]{ISW2025}, together with the uniformity estimate proved in~\cite{BSW2022}, implies Theorem~\ref{thm-global:const}, Theorem~\ref{thm-global:sub1-const}, and Theorem~\ref{thm-global:x-const}.
In addition, if the uniformity estimate for integer polynomials with fixed $x^{n - 1}$-coefficient is obtained, then we would also get the global density of integer polynomials with fixed $x^{n - 1}$-coefficient and prime constant coefficient.
Note that Theorem~\ref{thm-local:copconst} and Theorem~\ref{thm-local:sub1-copconst} only imply that the global densities in Theorem~\ref{thm-global:const}-\ref{thm-global:x-const} are positive when $n \geq 3$, but the case $n = 2$ follows by direct calculation; if $p$ is odd, then $P^{\sqf}(V_2(\Z_p)_{p \nmid a_2}) = 1 - 1/p^2$ and $P^{\sqf}(V_2(\Z_p)_{a_1 = 1, p \nmid a_2}) = 1 - 1/(p^2 - p)$.

It remains to do the local density calculations and prove Theorem~\ref{thm-local:sub1}-\ref{thm-local:sub1-copconst}.
Let $V_n(\Z_p)$ denote the set of degree $n$ monic polynomials over $\Z_p$, and let $\Sigma$ be a subset of $V_n(\Z_p)$ defined by congruence conditions modulo $p^2$.
The densities $P^{\sqf}(\Sigma)$ and $P^{\max}(\Sigma)$ can be written in terms of the quantity
\[ \lambda_p(\{f \in \Sigma : \overline{f} = u\}) \]
    across all $u \in \F_p[x]_m$ of degree $n$, where $\F_p[x]_m$ denotes the set of monic polynomials over $\F_p$ and $\lambda_p$ is the $p$-adic Haar measure on $V_n(\Z_p)$.
Our key idea to compute $P^{\sqf}(\Sigma)$ and $P^{\max}(\Sigma)$ is by decomposing $\lambda_p(\{f \in \Sigma : \overline{f} = u\})$ as a linear combination of completely multiplicative functions.
We write $\lambda_p(\{f \in \Sigma : \overline{f} = u\})$ as $w(\psi(u))$, where $\psi : \F_p[x]_m \to G$ is a monoid homomorphism to some finite abelian group $G$ and $w : G \to \C$ is a weight function.
Using Fourier transform, we express the weight function $w$ as a linear combination of the characters of $G$.
Composing these characters with $\psi$ gives completely multiplicative functions, thus giving the desired decomposition of $\lambda_p(\{f \in \Sigma : \overline{f} = u\})$.

If the constant coefficient of the polynomials in $\Sigma$ are units, then $\lambda_p(\{f \in \Sigma : \overline{f} = u\}) = 0$ whenever $x \mid u$.
Thus, we would need to change the domain of $\psi$ from $\F_p[x]_m$ to $\F_p[x]_{m, x \nmid u}$, the set of monic polynomials over $\F_p$ not divisible by $x$.
This causes the main difference between the densities of $V_n(\Z_p)_{p \nmid a_n}$ as seen in Theorem~\ref{thm-local:copconst} and the densities of $V_n(\Z_p)$ in the $n$-limit.
The rest of the argument work as before.

We note that our method bears a similarity to the method of~\cite{ABZ2007} for the case $\Sigma = V_n(\Z_p)$, in which we utilize the theory of generating series of arithmetic functions over function fields, with different executions.

In \S\ref{section:density-formula}, we prove formulas for the squarefree discriminant density and maximality density of a subset of $V_n(\Z_p)$ using Fourier analysis over finite abelian groups.
In \S\ref{section:sub1} and \S\ref{section:sub2}, we apply the formulas proved in \S\ref{section:density-formula} to prove Theorem~\ref{thm-local:sub1} and Theorem~\ref{thm-local:sub2}, respectively.
In \S\ref{section:density-formula-unit}, we modify the formulas proved in \S\ref{section:density-formula} to work over the subsets of $V_n(\Z_p)$ containing only polynomials whose constant coefficient is a unit.
In \S\ref{section:copconst}, \S\ref{section:fixconst}, and \S\ref{section:sub1-copconst}, we apply these modified formulas to prove Theorem~\ref{thm-local:copconst}, Theorem~\ref{thm-local:fixconst}, and Theorem~\ref{thm-local:sub1-copconst}, respectively.

\subsection{Common notations}

Before we proceed, we set up some common notations to be used throughout the paper.

For any finite abelian group $G$, the \emph{dual group} $\widehat{G}$ of $G$ is the set of group homomorphisms from $G$ to the complex unit circle $\{z \in \C : |z| = 1\}$.
The elements of $\widehat{G}$ are called the \emph{characters} of $G$.
For any function $w : G \to \C$, the \emph{Fourier transform} of $w$ is the function $\hat{w} : \widehat{G} \to \C$ defined by the formula
\[ \hat{w}(\chi) = \frac{1}{\# G} \sum_{\gamma \in G} w(\gamma) \chi(\gamma)^{-1}. \]

We denote by $\F_p[x]_m$ the set of monic polynomials over $\F_p$.
For any function $\rho : S \to \C$ on a subset $S \subseteq \F_p[x]_m$, the \emph{generating series} associated to $\rho$ is the power series
\[ L_{\rho}(T) = \sum_{u \in S} \rho(u) \, T^{\deg(u)}. \]
In \S\ref{section:density-formula}-\ref{section:sub2}, we use the above notation with $S = \F_p[x]_m$, while in \S\ref{section:density-formula-unit}-\ref{section:sub1-copconst}, we use the above notation with $S = \F_p[x]_{m, x \nmid u}$, the set of $u \in \F_p[x]_m$ such that $x \nmid u$.
Finally, for any power series $L(T)$ and an integer $k \geq 0$, we denote by $[L(T)]_k$ the $T^k$-coefficient of $L(T)$.

\section{General density formula}\label{section:density-formula}

Fix an integer $n \geq 2$ and a prime number $p$.
Recall that $V_n(\Z_p)$ is the set of degree $n$ monic polynomials over $\Z_p$ and $\lambda_p$ is the $p$-adic Haar measure on $V_n(\Z_p)$.
Let $\Sigma$ be a subset of $V_n(\Z_p)$ defined by congruence conditions mod $p^2$.
Let $P_0^{\sqf}(\Sigma)$ and $P_1^{\sqf}(\Sigma)$ denote the density of polynomials in $\Sigma$ whose discriminant has $p$-adic valuation $0$ and $1$, respectively.
Let $P^{\max}(\Sigma)$ denote the density of polynomials $f \in \Sigma$ such that $\Z_p[x]/(f(x))$ is the maximal order of $\Q_p[x]/(f(x))$.
In this section, we prove general formulas for $P_0^{\sqf}(\Sigma)$, $P_1^{\sqf}(\Sigma)$, and $P^{\max}(\Sigma)$.
We first set up some technical definitions.

For any non-negative integer $k < n$, we say that the $x^k$-coefficient of (polynomials in) $\Sigma$ is \emph{defined by congruence conditions mod $p$} if for any $f \in \Sigma$ and $c \in \Z_p$ such that $p \mid c$, we have $f + cx^k \in \Sigma$.
Next, we define \emph{$\Sigma$-admissible triples}, which encode the distribution of $\lambda_p(\{f \in \Sigma : \overline{f} = u\})$ across all $u \in \F_p[x]_m$ of degree $n$.

\begin{definition}\label{def:admissible}
A \emph{$\Sigma$-admissible triple} is a triple $(G, \psi, w)$, where $G$ is a finite abelian group, $\psi : \F_p[x]_m \to G$ is a monoid homomorphism, and $w : G \to \C$ is a function such that for any $u \in \F_p[x]_m$ of degree $n$,
\[ w(\psi(u)) = \frac{\lambda_p(\{f \in \Sigma : \overline{f} = u\})}{\lambda_p(\Sigma)}. \]
\end{definition}

The general formulas, which we wish to prove in this section, are as follows.

\begin{theorem}\label{thm-density:sqf-0}
Let $(G, \psi, w)$ be a $\Sigma$-admissible triple.
Then \[ P_0^{\sqf}(\Sigma) = \sum_{\chi \in \widehat{G}} \hat{w}(\chi) \left[\frac{L_{\chi \circ \psi}(T)}{L_{\chi^2 \circ \psi}(T^2)}\right]_n. \]
\end{theorem}

\begin{theorem}\label{thm-density:sqf-1}
Suppose that the constant coefficient of $\Sigma$ is defined by congruence conditions mod $p$.
Let $(G, \psi, w)$ be a $\Sigma$-admissible triple.
If $p = 2$, then $P_1^{\sqf}(\Sigma) = 0$, and if $p$ is odd, then
\[ P_1^{\sqf}(\Sigma) = \left(1 - \frac{1}{p}\right) \sum_{\chi \in \widehat{G}} \hat{w}(\chi) \left[\sum_{c \in \F_p} \frac{\chi(\psi(x + c))^2}{1 + \chi(\psi(x + c)) \, T} \cdot \frac{L_{\chi \circ \psi}(T)}{L_{\chi^2 \circ \psi}(T^2)}\right]_{n - 2}. \]
\end{theorem}

\begin{theorem}\label{thm-density:max}
Suppose that the $x^k$-coefficient of $\Sigma$ is defined by congruence conditions mod $p$ for all $k < \lfloor n/2 \rfloor$.
Let $(G, \psi, w)$ be a $\Sigma$-admissible triple.
Then \[ P^{\max}(\Sigma) = \sum_{\chi \in \widehat{G}} \hat{w}(\chi) \left[\frac{L_{\chi \circ \psi}(T)}{L_{\chi^2 \circ \psi}(T^2/p)}\right]_n. \]
\end{theorem}

We now discuss how to apply the above formulas by considering the general case, where $\Sigma = V_n(\Z_p)$.
It is easy to see that
\[ \frac{\lambda_p(\{f \in V_n(\Z_p) : \overline{f} = u\})}{\lambda_p(V_n(\Z_p))} = \frac{1}{p^n}. \]
Indeed, by reduction modulo $p$, the left hand side counts the probability that a monic polynomial over $\F_p$ of degree $n$ is equal to $u$.
Thus, a $V_n(\Z_p)$-admissible triple is a triple $(G, \psi, w)$ with the weight function $w$ satisfying $w(\psi(u)) = 1/p^n$ for all $u \in \F_p[x]_m$ of degree $n$.
Picking $G = \{1\}$ and $w$ to be the constant function with value $1/p^n$ (and $\psi$ to be the trivial homomorphism) gives us a $V_n(\Z_p)$-admissible triple, and we have $\hat{w}(1) = w(1) = 1/p^n$ since $G = \{1\}$.
The dual group $\widehat{G}$ consists of only the trivial character $1$, and $1 \circ \psi = \mathbf{1}$ is the all-one function from $\F_p[x]_m$ to $\C$.
Plugging the formula of $\hat{w}$ into the above formulas and evaluating $L_{\mathbf{1}}(T) = (1 - pT)^{-1}$ gives us the value of $P_0^{\sqf}(\Sigma)$, $P_1^{\sqf}(\Sigma)$, and $P^{\max}(\Sigma)$.

In the previous paragraph, when picking a $V_n(\Z_p)$-admissible triple $G$, we could have picked any finite group $G$; a general candidate is the group $G = 1 + y \F_p[y]/(y^n)$, with the monoid homomorphism $\psi : \F_p[x]_m \to G$ defined by
\[ \psi(x^k + a_1 x^{k - 1} + \ldots + a_k) = 1 + a_1 y + a_2 y^2 + \ldots + a_k y^k \bmod{y^n}. \]
However, taking $G = \{1\}$ allows us to write the densities of $\Sigma$ using sums with fewer terms and makes the computation of the Fourier transform $\hat{w}$ much easier (trivial in this case).
In our applications, this flexibility allows us to choose a relatively small $G$ and minimize the work needed to calculate the Fourier transform $\hat{w}$ and the generating series $L_{\chi \circ \psi}(T)$ across all characters $\chi \in \widehat{G}$.

Finally, our method does not work directly when the constant coefficients of all polynomials in $\Sigma$ are not divisible by $p$, which is needed to prove Theorem~\ref{thm-local:copconst}-\ref{thm-local:sub1-copconst}.
But we can slightly modify this method, where we replace the domain of $\psi : \F_p[x]_m \to G$ with the set $\F_p[x]_{m, x \nmid u}$ of monic polynomials over $\F_p$ that are not divisible by $x$; see \S\ref{section:density-formula-unit}.
In fact, this also allows us to handle the case where $\Sigma$ is defined by a finite set of divisibility conditions mod $p$, in addition to congruence conditions mod $p^2$.

\subsection{Proof of Theorem~\ref{thm-density:sqf-0} and Theorem~\ref{thm-density:sqf-1}}

We start with a criterion that determines if a monic polynomial over $\Z_p$ has squarefree discriminant.

\begin{lemma}\label{lem:sqfree-criterion}
(\cite[Proposition 6.7]{ABZ2007})
Let $p$ be a prime number and $f \in \Z_p[x]$ be a monic non-constant polynomial.
Denote by $\Delta(f)$ the discriminant of $f$, and by $\overline{f}$ the reduction of $f$ modulo $p$.
\begin{itemize}
    \item   $\Delta(f)$ is a unit if and only if $\overline{f} \in \F_p[x]$ is squarefree.
    \item   $\Delta(f)$ has valuation $1$ if and only if $p \neq 2$ and the following holds: there exists $c \in \F_p$ and $u \in \F_p[x]_m$ squarefree such that $x + c \nmid u$, $\overline{f} = (x + c)^2 u$, and $p^2 \nmid f(-\tilde{c})$ for some (any) lift $\tilde{c}$ of $c$ in $\Z_p$.
\end{itemize}
\end{lemma}

\begin{proof}[Proof of Theorem~\ref{thm-density:sqf-0}]
By applying Lemma~\ref{lem:sqfree-criterion} and the fact that $(G, \psi, w)$ is $\Sigma$-admissible, we get
\[ P_0^{\sqf}(\Sigma) = \sum_{\substack{u \in \F_p[x]_m \\ \deg(u) = n}} |\mu(u)| \; \frac{\lambda_p(\{f \in \Sigma : \overline{f} = u\})}{\lambda_p(\Sigma)} = \sum_{\substack{u \in \F_p[x]_m \\ \deg(u) = n}} |\mu(u)| \; w(\psi(u)). \]
Applying Fourier expansion on $w$ gives
\[ P_0^{\sqf}(\Sigma) = \sum_{\chi \in \widehat{G}} \hat{w}(\chi) \sum_{\substack{u \in \F_p[x]_m \\ \deg(u) = n}} |\mu(u)| \; \chi(\psi(u)) = \sum_{\chi \in \widehat{G}} \hat{w}(\chi) \left[L_{|\mu| (\chi \circ \psi)}(T)\right]_n = \sum_{\chi \in \widehat{G}} \hat{w}(\chi) \left[\frac{L_{\chi \circ \psi}(T)}{L_{\chi^2 \circ \psi}(T^2)}\right]_n, \]
    where the last equality holds since $\chi \circ \psi$ is completely multiplicative.
\end{proof}

Now we prove Theorem~\ref{thm-density:sqf-1}.
Lemma~\ref{lem:sqfree-criterion} gives $P_1^{\sqf}(\Sigma) = 0$ if $p = 2$, so it remains to consider the case where $p$ is odd.
We start by counting, for each $c \in \F_p$ and $u \in \F_p[x]_m$, the "size" of the set of monic lifts $f \in \Sigma$ of $(x + c)^2 u$ such that $p^2 \nmid f(-\tilde{c})$.

\begin{proposition}\label{prop-density:sqf-1-lift-count}
Suppose that the constant coefficient of $\Sigma$ is defined by congruence conditions mod $p$.
Then for any $c \in \F_p$ and $u \in \F_p[x]_m$ of degree $n - 2$, letting $\tilde{c}$ be an arbitrary $p$-adic lift of $c$, we have
\[ \lambda_p\left(\{f \in \Sigma : \overline{f} = (x + c)^2 u, \, p^2 \nmid f(-\tilde{c})\}\right) = \left(1 - \frac{1}{p}\right) \lambda_p\left(\{f \in \Sigma : \overline{f} = (x + c)^2 u\}\right). \]
\end{proposition}
\begin{proof}
Fix the non-constant coefficients of $f$.
Then the constant coefficient of $\overline{f}$ has $p$ possible lifts mod $p^2$, and $p^2 \mid f(-\tilde{c})$ holds for exactly one of them.
\end{proof}

\begin{proof}[Proof of Theorem~\ref{thm-density:sqf-1}]
Assume that $p$ is odd.
For any $c \in \F_p$, define the function $\mathbf{1}_{x + c} : \F_p[x]_m \to \C$ as follows: $\mathbf{1}_{x + c}(u)$ is equal to $1$ if $x + c \nmid u$ and $0$ if $x + c \mid u$.
For any monic polynomial over $\F_p$ of form $(x + c)^2 u$ with $u$ squarefree, there is exactly one possible choice for $c$, since such polynomials cannot be divisible by $(x + c)^2$ for more than one value of $c$.
As a result, Lemma~\ref{lem:sqfree-criterion} yields
\[ P_1^{\sqf}(\Sigma) = \sum_{c \in \F_p} \sum_{\substack{u \in \F_p[x]_m \\ \deg(u) = n - 2}} \mathbf{1}_{x + c}(u) \, |\mu(u)| \; \frac{\lambda_p(\{f \in \Sigma : \overline{f} = (x + c)^2 u, \, p^2 \nmid f(-\tilde{c})\})}{\lambda_p(\Sigma)}. \]
Applying Proposition~\ref{prop-density:sqf-1-lift-count} and the fact that $(G, \psi, w)$ is $\Sigma$-admissible respectively gives
\begin{align*}
    P_1^{\sqf}(\Sigma)
    &= \left(1 - \frac{1}{p}\right) \sum_{c \in \F_p} \sum_{\substack{u \in \F_p[x]_m \\ \deg(u) = n - 2}} \mathbf{1}_{x + c}(u) \, |\mu(u)| \; \frac{\lambda_p(\{f \in \Sigma : \overline{f} = (x + c)^2 u\})}{\lambda_p(\Sigma)} \\
    &= \left(1 - \frac{1}{p}\right) \sum_{c \in \F_p} \sum_{\substack{u \in \F_p[x]_m \\ \deg(u) = n - 2}} \mathbf{1}_{x + c}(u) \, |\mu(u)| \; w(\psi((x + c)^2 u)).
\end{align*}
Next, applying Fourier expansion on $w$ gives
\begin{align*}
    P_1^{\sqf}(\Sigma)
    &= \left(1 - \frac{1}{p}\right) \sum_{\chi \in \widehat{G}} \hat{w}(\chi) \sum_{c \in \F_p} \sum_{\substack{u \in \F_p[x]_m \\ \deg(u) = n - 2}} \mathbf{1}_{x + c}(u) \, |\mu(u)| \; \chi(\psi((x + c)^2 u)) \\
    &= \left(1 - \frac{1}{p}\right) \sum_{\chi \in \widehat{G}} \hat{w}(\chi) \sum_{c \in \F_p} \chi(\psi(x + c))^2 \sum_{\substack{u \in \F_p[x]_m \\ \deg(u) = n - 2}} \mathbf{1}_{x + c}(u) \, |\mu(u)| \; \chi(\psi(u)) \\
    &= \left(1 - \frac{1}{p}\right) \sum_{\chi \in \widehat{G}} \hat{w}(\chi) \sum_{c \in \F_p} \chi(\psi(x + c))^2 \left[L_{\mathbf{1}_{x + c} \, |\mu| \, (\chi \circ \psi)}(T)\right]_{n - 2}.
\end{align*}
For each $\chi \in \widehat{G}$, since $\chi \circ \psi$ is completely multiplicative, we get
\[ L_{\mathbf{1}_{x + c} \, |\mu| \, (\chi \circ \psi)}(T) = \frac{L_{|\mu| \, (\chi \circ \psi)}(T)}{1 + \chi(\psi(x + c)) \, T} = \frac{1}{1 + \chi(\psi(x + c)) \, T} \cdot \frac{L_{\chi \circ \psi}(T)}{L_{\chi^2 \circ \psi}(T^2)}. \]
Applying this equality to the previous equality yields the desired formula.
\end{proof}

\subsection{Proof of Theorem~\ref{thm-density:max}}

We start with a criterion that determines, given a monic polynomial $f \in \Z_p[x]$, whether $\Z_p[x]/(f(x))$ is the maximal order of $\Q_p[x]/(f(x))$.

\begin{lemma}[Dedekind's criterion]\label{lem-density:max-criterion}
(\cite[Corollary 3.2]{ABZ2007})
Fix a prime number $p$.
Let $f$ be a monic polynomial over $\Z_p$.
Then the ring $\Z_p[x]/(f(x))$ is the maximal order in $\Q_p[x]/(f(x))$ if and only if $f \notin (p, g)^2$ for any $g \in \Z_p[x]$ monic such that $\overline{g}$ is irreducible, where $(p, g)$ is the ideal of $\Z_p[x]$ generated by $p$ and $g$.
\end{lemma}

\begin{lemma}\label{lem-density:max-intersect}
(\cite[Lemma 3.3]{ABZ2007})
Fix $g, h \in \Z_p[x]$ monic and coprime modulo $p$.
Then \[ (p, g)^2 \cap (p, h)^2 = (p, gh)^2. \]
\end{lemma}

The following lemma follows from Lemma~\ref{lem-density:max-criterion} and Lemma~\ref{lem-density:max-intersect} using the inclusion-exclusion sieve.

\begin{lemma}\label{lem-density:max-overgeneral}
We have
\[ P^{\max}(\Sigma) = \sum_{u \in \F_p[x]_m} \mu(u) \; \frac{\lambda_p(\Sigma \cap (p, \tilde{u})^2)}{\lambda_p(\Sigma)}. \]
\end{lemma}

To simplify the above formula, we prove the following proposition.

\begin{proposition}\label{prop-density:max-reduce}
Suppose that the $x^k$-coefficient of $\Sigma$ is defined by congruence conditions mod $p$ for all non-negative integers $k < \lfloor n/2 \rfloor$.
Then for any $u \in \F_p[x]_m$, we have
\[ \lambda_p(\Sigma \cap (p, \tilde{u})^2) = p^{-\deg(u)} \, \lambda_p(\{f \in \Sigma : u^2 \mid \overline{f}\}). \]
\end{proposition}
\begin{proof}
Fix $u \in \F_p[x]_m$, and denote $d = \deg(u)$.
If $d > \lfloor n/2 \rfloor$, then both sides are zero, so now assume that $d \leq \lfloor n/2 \rfloor$.
Let $\Sigma_{\bmod p^2}$ denote the set of polynomials over $\Z/p^2 \Z$ of the form $f \bmod{p^2}$ for some $f \in \Sigma$.
Since $\Sigma$ is defined by congruence conditions mod $p^2$, we have
\[ \lambda_p(\Sigma \cap (p, \tilde{u})^2) = p^{-2n} \#(\Sigma_{\bmod p^2} \cap (p, \tilde{u})^2) \;\; \text{ and } \;\; \lambda_p(\{f \in \Sigma : u^2 \mid \overline{f}\}) = p^{-2n} \#\{f \in \Sigma_{\bmod p^2} : u^2 \mid \overline{f}\}. \]
Thus, it suffices to show that
\[ \#(\Sigma_{\bmod p^2} \cap (p, \tilde{u})^2) = p^{-d} \# \{f \in \Sigma_{\bmod p^2} : u^2 \mid \overline{f}\}. \]

Let $A_d$ be the set of polynomials over $\Z/p^2 \Z$ of degree less than $d$ with coefficients divisible by $p$.
Consider the map $\phi : \left(\Sigma_{\bmod p^2} \cap (p, \tilde{u})^2\right) \times A_d \to \{f \in \Sigma_{\bmod p^2} : u^2 \mid \overline{f}\}$ defined by
\[ \phi(f, g) = f + g. \]
Since $(p, \tilde{u})^2 \subseteq (p, \tilde{u}^2)$ and the $x^k$-coefficient of $\Sigma$ is defined by congruence conditions mod $p$ for all $k < d$, the map $\phi$ is well-defined.
Since $|A_d| = p^d$, it now suffices to show that $\phi$ is bijective.

First, suppose that $f_1 + g_1 = f_2 + g_2$ for some $f_1, f_2 \in \Sigma_{\bmod p^2} \cap (p, \tilde{u})^2$ and $g_1, g_2 \in A_d$.
Then we have $g_2 - g_1 = f_1 - f_2 \in (p, \tilde{u})^2 \cap (p) = (p \tilde{u})$.
Since $g_1, g_2 \in A_d$, we have $\deg(g_2 - g_1) < d = \deg(u)$.
Thus, we get $g_1 = g_2$ and $f_1 = f_2$.
This proves that $\phi$ is injective.

Now we prove that $\phi$ is surjective.
Fix any $h \in \Sigma_{\bmod p^2}$ such that $u^2 \mid \overline{h}$.
Then there exists $f_0, g_0 \in \F_p[x]$ such that $h = \tilde{u}^2 f_0 + p \tilde{g_0}$, where $\tilde{g_0}$ is an arbitrary lift of $g_0$ over $\Z/p^2 \Z$.
By division algorithm, we can write $g_0 = ug_1 + g_2$ for some $g_1, g_2 \in \F_p[x]$ such that $\deg(g_2) < \deg(u) = d$.
Then we have
\[ h = \tilde{u}^2 f_0 + p \tilde{g_0} = \tilde{u}^2 f_0 + p \tilde{u} \tilde{g_1} + p \tilde{g_2}, \]
    where $h - p \tilde{g_2} = \tilde{u}^2 f_0 + p \tilde{u} \tilde{g_1} \in (p, \tilde{u})^2$ and $p \tilde{g_2} \in A_d$ since $\deg(g_2) < d$.
Furthermore, since the $x^k$-coefficient of $\Sigma$ is defined by congruence conditions mod $p$ for all non-negative integers $k < d$, $p \tilde{g_2} \in A_d$ implies $h - p \tilde{g_2} \in \Sigma_{\bmod p^2}$, and so $h - p \tilde{g_2} \in \Sigma_{\bmod p^2} \cap (p, \tilde{u})^2$.
This shows that $\phi$ is surjective.
\end{proof}

\begin{proof}[Proof of Theorem~\ref{thm-density:max}]
By Lemma~\ref{lem-density:max-overgeneral} and Proposition~\ref{prop-density:max-reduce}, we get
\[ P^{\max}(\Sigma) = \sum_{u \in \F_p[x]_m} \frac{\mu(u)}{p^{\deg(u)}} \cdot \frac{\lambda_p(\{f \in \Sigma : u^2 \mid \overline{f}\})}{\lambda_p(\Sigma)} = \sum_{\substack{u, v \in \F_p[x]_m \\ 2 \deg(u) + \deg(v) = n}} \frac{\mu(u)}{p^{\deg(u)}} \cdot \frac{\lambda_p(\{f \in \Sigma : \overline{f} = u^2 v\})}{\lambda_p(\Sigma)}. \]
Applying the fact that $(G, \psi, w)$ is $\Sigma$-admissible and then applying Fourier expansion on $w$ gives
\[ P^{\max}(\Sigma) = \sum_{\substack{u, v \in \F_p[x]_m \\ 2 \deg(u) + \deg(v) = n}} \frac{\mu(u)}{p^{\deg(u)}} \; w(\psi(u^2 v)) = \sum_{\chi \in \widehat{G}} \hat{w}(\chi) \sum_{\substack{u, v \in \F_p[x]_m \\ 2 \deg(u) + \deg(v) = n}} \frac{\mu(u)}{p^{\deg(u)}} \; \chi(\psi(u^2 v)). \]
Rearranging gives
\[ P^{\max}(\Sigma) = \sum_{\chi \in \widehat{G}} \hat{w}(\chi) \sum_{\substack{u, v \in \F_p[x]_m \\ 2 \deg(u) + \deg(v) = n}} \frac{\mu(u) \chi(\psi(u))^2}{p^{\deg(u)}} \; \chi(\psi(v)) = \sum_{\chi \in \widehat{G}} \hat{w}(\chi) \left[L_{\mu \cdot (\chi^2 \circ \psi)}(T^2/p) \; L_{\chi \circ \psi}(T)\right]_n. \]
Finally, for each $\chi \in \widehat{G}$, since $\chi^2 \circ \psi$ is completely multiplicative, we have $L_{\mu \cdot (\chi^2 \circ \psi)}(T^2/p) = L_{\chi^2 \circ \psi}(T^2/p)^{-1}$.
Substituting into the above equality gives the desired formula.
\end{proof}

\section{Proof of Theorem~\ref{thm-local:sub1}}\label{section:sub1}

Let $n \geq 2$ and $p$ be a prime number.
Fix $b_1 \in \Z_p$, and recall that
\[ V_n(\Z_p)_{a_1 = b_1} = \{f(x) = x^n + a_1 x^{n - 1} + \ldots + a_n \in \Z_p[x] : a_1 = b_1\}. \]
Then Theorem~\ref{thm-local:sub1} follows from the following exact formula for the densities of $V_n(\Z_p)_{a_1 = b_1}$.

\begin{theorem}\label{thm-sub1:main}
For any $n \geq 2$ and prime number $p$,
\begin{align}
    P^{\max}(V_n(\Z_p)_{a_1 = b_1}) &= \begin{cases} 1 & \text{if } (n, p) = (2, 2) \text{ and } 2 \nmid b_1, \\ 1/2 & \text{if } (n, p) = (2, 2) \text{ and } 2 \mid b_1, \\ 1 - 1/p^2 & \text{otherwise,} \end{cases} \label{eq-sub1:max} \\
    P_0^{\sqf}(V_n(\Z_p)_{a_1 = b_1}) &= \begin{cases} 1 & \text{if } (n, p) = (2, 2) \text{ and } 2 \nmid b_1, \\ 0 & \text{if } (n, p) = (2, 2) \text{ and } 2 \mid b_1, \\ 1 - 1/p & \text{otherwise.} \end{cases} \label{eq-sub1:sqf-0}
\end{align}
We have $P_1^{\sqf}(V_n(\Z_2)_{a_1 = b_1}) = 0$ and if $p$ is odd, then
\begin{equation}\label{eq-sub1:sqf-1}
    P_1^{\sqf}(V_n(\Z_p)_{a_1 = b_1}) = \begin{cases} (p - 1)/p^2 & \text{if } n = 2, \\ \dfrac{(p - 1)^2}{p^2 (p + 1)} - \dfrac{(-1)^n (p - 1)^2}{p^n (p + 1)} + \dfrac{(-1)^n}{p^n} (p - 1) \iota_{n, p}(\overline{b_1}) & \text{if } n > 2, \end{cases}
\end{equation}
    where for any $b \in \F_p$,
\begin{equation}\label{eq-sub1:iota}
    \iota_{n, p}(b) = \frac{1}{p} \sum_{\chi \in \widehat{\F_p} \setminus \{1\}} \chi(b)^{-1} \sum_{c \in \F_p} \chi(c)^n = \begin{cases} 0 & \text{if } p \nmid n, \\ -1 & \text{if } p \mid n \text{ and } b \neq 0, \\ p - 1 & \text{if } p \mid n \text{ and } b = 0. \end{cases}
\end{equation}
\end{theorem}

It remains to prove Theorem~\ref{thm-sub1:main}.
First, since squarefree discriminant and maximality are mod $p^2$ conditions, the densities of $V_n(\Z_p)_{a_1 = b_1}$ are the same as those of the set
\[ V_n(\Z_p)_{a_1 \equiv b_1 \bmod{p^2}} = \{f(x) = x^n + a_1 x^{n - 1} + \ldots + a_n \in \Z_p[x] : a_1 \equiv b_1 \bmod{p^2}\}, \]
    which is defined by congruence conditions mod $p^2$.
We start by finding a $V_n(\Z_p)_{a_1 \equiv b_1 \bmod{p^2}}$-admissible triple $(G, \psi, w)$ as defined in Definition~\ref{def:admissible} and computing $\hat{w}(\chi)$ and $L_{\chi \circ \psi}(T)$ for all $\chi \in \widehat{G}$.
Then we apply the formulas obtained in \S\ref{section:density-formula} to prove Theorem~\ref{thm-sub1:main}.

Define the monoid homomorphism $\varphi_1 : \F_p[x]_m \to \F_p$ by the formula
\begin{equation}\begin{split}\label{eq:varphi1-def}
    \varphi_1(x^n + a_1 x^{n - 1} + \ldots + a_n) &= a_1, \\
    \varphi_1(1) &= 0.
\end{split}\end{equation}
Then it is easy to see that for any $u \in \F_p[x]_m$ of degree $n$,
\[ \frac{\lambda_p(\{f \in V_n(\Z_p)_{a_1 \equiv b_1 \bmod{p^2}} : \overline{f} = u\})}{\lambda_p(V_n(\Z_p)_{a_1 \equiv b_1 \bmod{p^2}})} = \begin{cases} 1/p^{n - 1} & \text{if } \varphi_1(u) = \overline{b_1}, \\ 0 & \text{if } \varphi_1(u) \neq \overline{b_1}. \end{cases} \]
Thus, $(\F_p, \varphi_1, w)$ is $V_n(\Z_p)_{a_1 \equiv b_1 \bmod{p^2}}$-admissible, where the weight function $w : \F_p \to \C$ is defined by
\[ w(c) = \begin{cases} 1/p^{n - 1} & \text{if } c = \overline{b_1}, \\ 0 & \text{if } c \neq \overline{b_1}. \end{cases} \]
By direct computation, the Fourier transform $\hat{w}$ of $w$ has the formula
\begin{equation}\label{eq-sub1:weightFT}
    \hat{w}(\chi) = \frac{1}{\# \F_p} \sum_{c \in \F_p} w(c) \chi(c)^{-1} = \frac{1}{p^n} \chi(\overline{b_1})^{-1}.
\end{equation}

Next, we compute $L_{\chi \circ \varphi_1}(T)$ for any $\chi \in \widehat{\F_p}$ such that $\chi \neq 1$.
Note that $1 \circ \varphi_1 = \mathbf{1}$, the all-one function on $\F_p[x]_m$.

\begin{lemma}\label{lem-sub1:char-gen}
For any $\chi \in \widehat{\F_p}$ such that $\chi \neq 1$, we have $L_{\chi \circ \varphi_1}(T) = 1$.
\end{lemma}
\begin{proof}
Clearly, we have $[L_{\chi \circ \varphi_1}(T)]_0 = 1$.
Now fix some $n \geq 1$.
For each $c \in \F_p$, there exists exactly $p^{n - 1}$ polynomials $u \in \F_p[x]_m$ of degree $n$ whose $x^{n - 1}$-coefficient $\varphi_1(u)$ is equal to $c$.
Thus, expanding gives
\[ [L_{\chi \circ \varphi_1}(T)]_n = \sum_{\substack{u \in \F_p[x]_m \\ \deg(u) = n}} \chi(\varphi_1(u)) = p^{n - 1} \sum_{c \in \F_p} \chi(c), \]
    which is zero, since $\chi \neq 1$.
\end{proof}

We are now ready to prove Theorem~\ref{thm-sub1:main}.
We start by computing the densities $P^{\max}(V_n(\Z_p)_{a_1 = b_1})$ and $P_0^{\sqf}(V_n(\Z_p)_{a_1 = b_1})$.
Using the formula~\eqref{eq-sub1:weightFT} for $\hat{w}$, Theorem~\ref{thm-density:max} and Theorem~\ref{thm-density:sqf-0} respectively implies
\begin{align}
    P^{\max}(V_n(\Z_p)_{a_1 = b_1}) &= \frac{1}{p^n} \sum_{\chi \in \widehat{\F_p}} \chi(\overline{b_1})^{-1} \left[\frac{L_{\chi \circ \varphi_1}(T)}{L_{\chi^2 \circ \varphi_1}(T^2/p)}\right]_n, \label{eq-sub1:max-init} \\
    P_0^{\sqf}(V_n(\Z_p)_{a_1 = b_1}) &= \frac{1}{p^n} \sum_{\chi \in \widehat{\F_p}} \chi(\overline{b_1})^{-1} \left[\frac{L_{\chi \circ \varphi_1}(T)}{L_{\chi^2 \circ \varphi_1}(T^2)}\right]_n. \label{eq-sub1:sqf-0-init}
\end{align}
Then~\eqref{eq-sub1:max} and~\eqref{eq-sub1:sqf-0} follows by the following proposition.

\begin{proposition}
For any $n \geq 2$, $b \in \F_p$, and $z \in \C$,
\[ \frac{1}{p^n} \sum_{\chi \in \widehat{\F_p}} \chi(b)^{-1} \left[\frac{L_{\chi \circ \varphi_1}(T)}{L_{\chi^2 \circ \varphi_1}(zT^2)}\right]_n = \begin{cases} 1 & \text{if } (n, p) = (2, 2) \text{ and } b = 1, \\ 1 - z & \text{if } (n, p) = (2, 2) \text{ and } b = 0, \\ 1 - z/p & \text{otherwise.} \end{cases} \]
\end{proposition}
\begin{proof}
First we note that
\[ \left[\frac{L_{\mathbf{1}}(T)}{L_{\mathbf{1}}(zT^2)}\right]_n = \left[\frac{1 - pzT^2}{1 - pT}\right]_n = \left[1 + pT + \frac{(1 - z/p) p^2T^2}{1 - pT}\right]_n = \left(1 - \frac{z}{p}\right) p^n. \]
Thus we can write
\[ \frac{1}{p^n} \sum_{\chi \in \widehat{\F_p}} \chi(b)^{-1} \left[\frac{L_{\chi \circ \varphi_1}(T)}{L_{\chi^2 \circ \varphi_1}(zT^2)}\right]_n = 1 - \frac{z}{p} + \frac{1}{p^n} \sum_{\chi \in \widehat{\F_p} \setminus \{1\}} \chi(b)^{-1} \left[\frac{L_{\chi \circ \varphi_1}(T)}{L_{\chi^2 \circ \varphi_1}(zT^2)}\right]_n. \]

If $p$ is odd, for any $\chi \in \widehat{\F_p}$, $\chi \neq 1$ implies $\chi^2 \neq 1$.
Then Lemma~\ref{lem-sub1:char-gen} implies that the final term in the right hand side is zero since $[1]_n = 0$, thus proving the proposition for $p$ odd.
If $p = 2$, then $\widehat{\F_2} = \{1, \chi_2\}$, where $\chi_2 : \F_2 \to \C^{\times}$ is the character defined by $\chi_2(0) = 1$ and $\chi_2(1) = -1$.
Then Lemma~\ref{lem-sub1:char-gen} implies
\[ \sum_{\chi \in \widehat{\F_p}} \chi(b)^{-1} \left[\frac{L_{\chi \circ \varphi_1}(T)}{L_{\chi^2 \circ \varphi_1}(zT^2)}\right]_n = 1 - \frac{z}{2} + \frac{\chi_2(b)^{-1}}{2^n} \left[1 - 2zT^2\right]_n = \begin{cases} 1 & \text{if } n = 2 \text{ and } b = 1, \\ 1 - z & \text{if } n = 2 \text{ and } b = 0, \\ 1 - z/2 & \text{otherwise.} \end{cases} \]
This completes the proof of the proposition.
\end{proof}

It remains to compute $P_1^{\sqf}(V_n(\Z_p)_{a_1 = b_1})$ for $p$ odd.
By Theorem~\ref{thm-density:sqf-1} and the formula~\eqref{eq-sub1:weightFT} for $\hat{w}$,
\begin{align}
    P_1^{\sqf}(V_n(\Z_p)_{a_1 = b_1})
    &= \frac{1 - 1/p}{p^n} \sum_{\chi \in \widehat{\F_p}} \chi(\overline{b_1})^{-1} \left[\sum_{c \in \F_p} \frac{\chi(\varphi_1(x + c))^2}{1 + \chi(\varphi_1(x + c)) \, T} \cdot \frac{L_{\chi \circ \varphi_1}(T)}{L_{\chi^2 \circ \varphi_1}(T^2)}\right]_{n - 2} \notag \\
    &= \frac{1 - 1/p}{p^n} \sum_{\chi \in \widehat{\F_p}} \chi(\overline{b_1})^{-1} \left[\sum_{c \in \F_p} \frac{\chi(c)^2}{1 + \chi(c) \, T} \cdot \frac{L_{\chi \circ \varphi_1}(T)}{L_{\chi^2 \circ \varphi_1}(T^2)}\right]_{n - 2}. \label{eq-sub1:sqf-1-init}
\end{align}
Since $p$ is odd, for any $\chi \in \widehat{\F_p}$, $\chi \neq 1$ implies $\chi^2 \neq 1$.
Thus applying Lemma~\ref{lem-sub1:char-gen} gives us
\begin{align*}
    P_1^{\sqf}(V_n(\Z_p)_{a_1 = b_1})
    &= \frac{1 - 1/p}{p^n} \left(\left[\frac{p}{1 + T} \cdot \frac{1 - pT^2}{1 - pT}\right]_{n - 2} + \sum_{\chi \in \widehat{\F_p} \setminus \{1\}} \chi(\overline{b_1})^{-1} \left[\sum_{c \in \F_p} \frac{\chi(c)^2}{1 + \chi(c) \, T}\right]_{n - 2}\right) \\
    &= \frac{p - 1}{p^n} \left[\frac{1 - pT^2}{(1 + T)(1 - pT)}\right]_{n - 2} + \frac{1 - 1/p}{p^n} \sum_{\chi \in \widehat{\F_p} \setminus \{1\}} \chi(\overline{b_1})^{-1} \sum_{c \in \F_p} (-1)^{n - 2} \chi(c)^n \\
    &= \frac{p - 1}{p^n} \left[1 + \frac{p - 1}{p + 1} \left(\frac{1}{1 - pT} - \frac{1}{1 + T}\right)\right]_{n - 2} + \frac{(-1)^n}{p^n} \cdot \frac{p - 1}{p} \sum_{\chi \in \widehat{\F_p} \setminus \{1\}} \chi(\overline{b_1})^{-1} \sum_{c \in \F_p} \chi(c)^n \\
    &= \frac{p - 1}{p^n} \left([1]_{n - 2} + \frac{p - 1}{p + 1} (p^{n - 2} - (-1)^{n - 2})\right) + \frac{(-1)^n}{p^n} (p - 1) \iota_{n, p}(\overline{b_1}),
\end{align*}
    where $\iota_{n, p}$ is as defined in~\eqref{eq-sub1:iota}.
If $n = 2$, then the above equality simplifies to $P_1^{\sqf}(V_2(\Z_p)_{a_1 = b_1}) = (p - 1)/p^2$ since $\iota_{2, p}(b) = 0$ for any $b \in \F_p$, thus proving~\eqref{eq-sub1:sqf-1} for $n = 2$.
If $n > 2$, then the above equality simplifies to
\[ P_1^{\sqf}(V_n(\Z_p)_{a_1 = b_1}) = \frac{(p - 1)^2}{p^2 (p + 1)} - \frac{(-1)^n (p - 1)^2}{p^n (p + 1)} + \frac{(-1)^n}{p^n} (p - 1) \iota_{n, p}(\overline{b_1}), \]
    proving~\eqref{eq-sub1:sqf-1} for $n > 2$.
This completes the proof of Theorem~\ref{thm-sub1:main}.

\section{Proof of Theorem~\ref{thm-local:sub2}}\label{section:sub2}

Let $n \geq 3$ and $p$ be a prime number.
Fix $b_1, b_2 \in \F_p$, and recall that
\begin{align*}
    V_n(\Z_p)_{a_1 = b_1, a_2 = b_2} &= \{f(x) = x^n + a_1 x^{n - 1} + \ldots + a_n \in \Z_p[x] : a_1 = b_1, a_2 = b_2\}, \\
    V_n(\Z_p)_{a_1 = b_1} &= \{f(x) = x^n + a_1 x^{n - 1} + \ldots + a_n \in \Z_p[x] : a_1 = b_1\}.
\end{align*}
We first state the exact formula for the densities of $V_n(\Z_p)_{a_1 = b_1, a_2 = b_2}$ in terms of the densities of $V_n(\Z_p)_{a_1 = b_1}$.
For each integer $n \geq 0$, odd prime number $p$, and $b \in \F_p$, define the quantity
\begin{equation}\label{eq-sub2:delta}
    \delta_{n, p}(b) = \begin{cases} -\left(\dfrac{-1}{p}\right)^{(n - 1)/2} p^{(n + 1)/2} & \text{if } 2 \nmid n \text{ and } b \neq 0, \\ \left(\dfrac{-1}{p}\right)^{(n - 1)/2} (p - 1) p^{(n + 1)/2} & \text{if } 2 \nmid n \text{ and } b = 0, \\ \left(\dfrac{b}{p}\right) \left(\dfrac{-1}{p}\right)^{n/2} p^{n/2 + 1} & \text{if } 2 \mid n, \end{cases}
\end{equation}
    and for any integer $t \geq 0$,
\begin{equation}\label{eq-sub2:H}
    H_{n, p, t}(b) = \sum_{\substack{k \leq n - 2 \\ p \nmid n - k}} (-1)^{\lceil k/2 \rceil} \left(\frac{n - k}{p}\right) \left(\frac{2}{p}\right)^{\lfloor k/2 \rfloor} \delta_{\lceil k/2 \rceil + t + 1, p}(b) + \sum_{\substack{0 \leq k \leq n - 2 \\ p \mid n - k}} (-1)^{\lceil k/2 \rceil} \left(\frac{2}{p}\right)^{\lfloor k/2 \rfloor} p \delta_{\lceil k/2 \rceil + t}(b).
\end{equation}
Then we prove the following formula.

\begin{theorem}\label{thm-sub2:main}
If $p = 2$, then $P_1^{\sqf}(V_n(\Z_2)_{a_1 = b_1, a_2 = b_2}) = 0$ and
\begin{align*}
    P^{\max}(V_n(\Z_2)_{a_1 = b_1, a_2 = b_2}) &= P^{\max}(V_n(\Z_2)_{a_1 = b_1}), \\
    P_0^{\sqf}(V_n(\Z_2)_{a_1 = b_1, a_2 = b_2}) &= P_0^{\sqf}(V_n(\Z_2)_{a_1 = b_1}).
\end{align*}

If $p$ is odd, $p \mid n$, and $p \nmid b_1$, then
\begin{align*}
    P^{\max}(V_n(\Z_p)_{a_1 = b_1, a_2 = b_2}) &= P^{\max}(V_n(\Z_p)_{a_1 = b_1}), \\
    P_0^{\sqf}(V_n(\Z_p)_{a_1 = b_1, a_2 = b_2}) &= P_0^{\sqf}(V_n(\Z_p)_{a_1 = b_1}), \\
    P_1^{\sqf}(V_n(\Z_p)_{a_1 = b_1, a_2 = b_2}) &= P_1^{\sqf}(V_n(\Z_p)_{a_1 = b_1}).
\end{align*}

If $p$ is odd, $p \mid n$, and $p \mid b_1$, then
\begin{align*}
    P^{\max}(V_n(\Z_p)_{a_1 = b_1, a_2 = b_2}) &= P^{\max}(V_n(\Z_p)_{a_1 = b_1}) + \frac{(-1)^{\lfloor n/2 \rfloor}}{p^{n + \lfloor n/2 \rfloor}} \left(\frac{2}{p}\right)^{\lfloor n/2 \rfloor} \delta_{\lceil n/2 \rceil + 1, p}(2 \overline{b_2}), \\
    P_0^{\sqf}(V_n(\Z_p)_{a_1 = b_1, a_2 = b_2}) &= P_0^{\sqf}(V_n(\Z_p)_{a_1 = b_1}) + \frac{(-1)^{\lfloor n/2 \rfloor}}{p^n} \left(\frac{2}{p}\right)^{\lfloor n/2 \rfloor} \delta_{\lceil n/2 \rceil + 1, p}(2 \overline{b_2}), \label{eq-sub2:sqf-0-odd-case3} \\
    P_1^{\sqf}(V_n(\Z_p)_{a_1 = b_1, a_2 = b_2}) &= P_1^{\sqf}(V_n(\Z_p)_{a_1 = b_1}) + \frac{(-1)^n (1 - 1/p)}{p^n} H_{n, p, 1}(2 \overline{b_2}).
\end{align*}

If $p$ is odd and $p \nmid n$, then
\begin{align*}
    P^{\max}(V_n(\Z_p)_{a_1 = b_1, a_2 = b_2}) &= P^{\max}(V_n(\Z_p)_{a_1 = b_1}) + \frac{(-1)^{\lfloor n/2 \rfloor}}{p^{n + \lfloor n/2 \rfloor}} \left(\frac{n}{p}\right) \left(\frac{2}{p}\right)^{\lfloor n/2 \rfloor} \delta_{\lceil n/2 \rceil, p}(2 \overline{b_2} - (1 - 1/\overline{n}) \overline{b_1}^2), \\
    P_0^{\sqf}(V_n(\Z_p)_{a_1 = b_1, a_2 = b_2}) &= P_0^{\sqf}(V_n(\Z_p)_{a_1 = b_1}) + \frac{(-1)^{\lfloor n/2 \rfloor}}{p^n} \left(\frac{n}{p}\right) \left(\frac{2}{p}\right)^{\lfloor n/2 \rfloor} \delta_{\lceil n/2 \rceil, p}(2 \overline{b_2} - (1 - 1/\overline{n}) \overline{b_1}^2), \\
    P_1^{\sqf}(V_n(\Z_p)_{a_1 = b_1, a_2 = b_2}) &= P_1^{\sqf}(V_n(\Z_p)_{a_1 = b_1}) + \frac{(-1)^n (1 - 1/p)}{p^n} \left(\frac{n}{p}\right) H_{n, p, 0}(2 \overline{b_2} - (1 - 1/\overline{n}) \overline{b_1}^2).
\end{align*}
\end{theorem}

We obtain Theorem~\ref{thm-local:sub2} from Theorem~\ref{thm-sub2:main} as follows.
The error term bound in the formulas given in Theorem~\ref{thm-local:sub2} follows from Theorem~\ref{thm-sub2:main}, Theorem~\ref{thm-local:sub1} on the asymptotic densities of $V_n(\Z_p)_{a_1 = b_1}$, and the estimate $\delta_{n, p}(b) \ll p^{\lceil n/2 \rceil + 1}$, which also implies
\[ H_{n, p, t}(b) \ll \sum_{\substack{0 \leq k \leq n - 2 \\ p \nmid n - k}} p^{\lceil (k + 2t + 2)/4 \rceil + 1} + \sum_{\substack{0 \leq k \leq n - 2 \\ p \mid n - k}} p^{\lceil (k + 2t)/4 \rceil + 2} \ll p^{\lceil (n + 2t)/4 \rceil + 1}. \]
Given integers $n \geq 4$, $b_1$, and $b_2$, the statement about when $P^{\max}(V_n(\Z_p)_{a_1 = b_1, a_2 = b_2}) = 1 - 1/p^2$ holds for all prime numbers $p$ follow from the formula $P^{\max}(V_n(\Z_p)_{a_1 = b_1}) = 1 - 1/p^2$ from Theorem~\ref{thm-local:sub1} and the fact that $\delta_{n, p}(b) = 0$ if and only if $n$ is even and $b = 0$.
For positivity of the densities, it suffices to check when $P_0^{\sqf}(V_n(\Z_p)_{a_1 = b_1, a_2 = b_2})$ is positive for $p$ odd.
By the bound $|\delta_{n, p}(b)| \leq (p - 1) p^{\lfloor n/2 \rfloor + 1}$, if $n \geq 4$, then
\[ |\delta_{\lceil n/2 \rceil + t, p}(b)| \leq (p - 1) p^{\lfloor (\lceil n/2 \rceil + t)/2 \rfloor + 1} \leq (p - 1) p^{\lceil n/4 \rceil + 1} < (p - 1) p^{n - 1}. \]
Since Theorem~\ref{thm-sub1:main} says $P_0^{\sqf}(V_n(\Z_p)_{a_1 = b_1}) = 1 - 1/p$, we get $P_0^{\sqf}(V_n(\Z_p)_{a_1 = b_1, a_2 = b_2}) > 0$ if $n \geq 4$.

We now prove Theorem~\ref{thm-sub2:main}.
First, since squarefree discriminant and maximality are mod $p^2$ conditions, the densities for $V_n(\Z_p)_{a_1 = b_1, a_2 = b_2}$ are the same as those of the set
\[ V_n(\Z_p)_{a_1 \equiv b_1 \bmod{p^2}, a_2 \equiv b_2 \bmod{p^2}} = \{f(x) = x^n + a_1 x^{n - 1} + \ldots + a_n \in \Z_p[x] : a_1 \equiv b_1 \bmod{p^2}\}, \]
    which is defined by congruence conditions mod $p^2$.
We start by finding a $V_n(\Z_p)_{a_1 \equiv b_1 \bmod{p^2}, a_2 \equiv b_2 \bmod{p^2}}$-admissible triple $(G, \psi, w)$ as defined in Definition~\ref{def:admissible} and computing $\hat{w}(\chi)$ and $L_{\chi \circ \psi}(T)$ for all $\chi \in \widehat{G}$.
Then we apply the formulas obtained in \S\ref{section:density-formula} to prove Theorem~\ref{thm-sub2:main}.

Define the group $G_p$ by
\[ G_p = 1 + y \F_p[y]/(y^3) = \{a_2 y^2 + a_1 y + 1 \bmod{y^3} : a_1, a_2 \in \F_p\}. \]
Define the monoid homomorphism $\varphi_2 : \F_p[x]_m \to G_p$ by the formula
\begin{equation}\begin{split}\label{eq:varphi2-def}
    \varphi_2(x^n + a_1 x^{n - 1} + \ldots + a_n) &= a_2 y^2 + a_1 y + 1, \\
    \varphi_2(x + c) &= cy + 1, \\
    \varphi_2(1) &= 1.
\end{split}\end{equation}
For convenience, we denote
\[ \gamma_0 = \overline{b_2} y^2 + \overline{b_1} y + 1 \in G_p. \]
Then it is easy to see that for any $u \in \F_p[x]_m$ of degree $n$,
\[ \frac{\lambda_p(\{f \in V_n(\Z_p)_{a_1 \equiv b_1 \bmod{p^2}, a_2 \equiv b_2 \bmod{p^2}} : \overline{f} = u\})}{\lambda_p(V_n(\Z_p)_{a_1 \equiv b_1 \bmod{p^2}, a_2 \equiv b_2 \bmod{p^2}})} = \begin{cases} 1/p^{n - 2} & \text{if } \varphi_2(u) = \gamma_0, \\ 0 & \text{if } \varphi_2(u) \neq \gamma_0. \end{cases} \]
Thus, $(G_p, \varphi_2, w)$ is $V_n(\Z_p)_{a_1 \equiv b_1 \bmod{p^2}, a_2 \equiv b_2 \bmod{p^2}}$-admissible, where $w : G_p \to \C$ is defined by
\[ w(\gamma) = \begin{cases} 1/p^{n - 2} & \text{if } \gamma = \gamma_0, \\ 0 & \text{if } \gamma \neq \gamma_0. \end{cases} \]
By direct computation, the Fourier transform $\hat{w}$ of $w$ has the formula
\begin{equation}\label{eq-sub2:weightFT}
    \hat{w}(\chi) = \frac{1}{\# G_p} \sum_{\gamma \in G_p} w(\gamma) \chi(\gamma)^{-1} = \frac{1}{p^n} \chi(\gamma_0)^{-1}.
\end{equation}

Next, we compute $L_{\chi \circ \varphi_2}(T)$ for any $\chi \in \widehat{G_p}$ with $\chi \neq 1$.
For convenience, for any $\chi \in \widehat{G_p}$, denote
\begin{equation}\label{eq-sub2:C-chi}
    C_{\chi} = \sum_{c \in \F_p} \chi(cy + 1).
\end{equation}

\begin{lemma}\label{lem-sub2:char-gen}
For any $\chi \in \widehat{G_p}$ with $\chi \neq 1$, we have $L_{\chi \circ \varphi_2}(T) = 1 + C_{\chi} T$.
\end{lemma}
\begin{proof}
Clearly, $[L_{\chi \circ \varphi_2}(T)]_0 = 1$, and $[L_{\chi \circ \varphi_2}(T)]_1 = C_{\chi}$ holds by definition.
Now fix some $n \geq 2$.
For any $\gamma \in G_p$, there exists exactly $p^{n - 2}$ polynomials $u \in \F_p[x]_m$ of degree $n$ such that $\varphi_2(u) = \gamma$.
Thus, we get
\[ [L_{\chi \circ \varphi_2}(T)]_n = \sum_{\substack{u \in \F_p[x]_m \\ \deg(u) = n}} \chi(\varphi_2(u)) = p^{n - 2} \sum_{\gamma \in G_p} \chi(\gamma), \]
    which is zero, since $\chi \neq 1$.
\end{proof}

We break down the proof of Theorem~\ref{thm-sub2:main} into two main steps.
First, when expressing the densities of $V_n(\Z_p)_{a_1 = b_1, a_2 = b_2}$ as a summation over $\widehat{G_p}$, we show that the terms corresponding to $\chi \in \widehat{G_p}$ with $\chi(y^2 + 1) = 1$ sum up to the corresponding densities of $V_n(\Z_p)_{a_1 = b_1}$, and we also simplify the contribution from $\chi \in \widehat{G_p}$ with $\chi(y^2 + 1) \neq 1$.
Second, we evaluate the contributions from $\chi \in \widehat{G_p}$ with $\chi(y^2 + 1) \neq 1$ using a general formula.
The first step is given by the following theorem.

\begin{theorem}\label{thm-sub2:step1}
If $p = 2$, then
\begin{align}
    P^{\max}(V_n(\Z_2)_{a_1 = b_1, a_2 = b_2}) &= P^{\max}(V_n(\Z_2)_{a_1 = b_1}), \label{eq-sub2:step1-max-two} \\
    P_0^{\sqf}(V_n(\Z_2)_{a_1 = b_1, a_2 = b_2}) &= P_0^{\sqf}(V_n(\Z_2)_{a_1 = b_1}). \label{eq-sub2:step1-sqf-0-two}
\end{align}
If $p$ is odd, then
\begin{align}
    P^{\max}(V_n(\Z_p)_{a_1 = b_1, a_2 = b_2}) &= P^{\max}(V_n(\Z_p)_{a_1 = b_1}) + \frac{(-1)^{\lfloor n/2 \rfloor}}{p^{n + \lfloor n/2 \rfloor}} \sum_{\substack{\chi \in \widehat{G_p} \\ \chi(y^2 + 1) \neq 1}} \chi(\gamma_0)^{-1} C_{\chi}^{n - 2 \lfloor n/2 \rfloor} C_{\chi^2}^{\lfloor n/2 \rfloor}, \label{eq-sub2:step1-max-odd} \\
    P_0^{\sqf}(V_n(\Z_p)_{a_1 = b_1, a_2 = b_2}) &= P_0^{\sqf}(V_n(\Z_p)_{a_1 = b_1}) + \frac{(-1)^{\lfloor n/2 \rfloor}}{p^n} \sum_{\substack{\chi \in \widehat{G_p} \\ \chi(y^2 + 1) \neq 1}} \chi(\gamma_0)^{-1} C_{\chi}^{n - 2 \lfloor n/2 \rfloor} C_{\chi^2}^{\lfloor n/2 \rfloor}, \label{eq-sub2:step1-sqf-0-odd} \\
    P_1^{\sqf}(V_n(\Z_p)_{a_1 = b_1, a_2 = b_2}) &= P_1^{\sqf}(V_n(\Z_p)_{a_1 = b_1}) \notag \\
        &\;+ \frac{(-1)^n (1 - 1/p)}{p^n} \sum_{k = 0}^{n - 2} (-1)^{\lceil k/2 \rceil} \sum_{\substack{\chi \in \widehat{G_p} \\ \chi(y^2 + 1) \neq 1}} \chi(\gamma_0)^{-1} C_{\chi^{n - k}} C_{\chi}^{k - 2 \lfloor k/2 \rfloor} C_{\chi^2}^{\lfloor k/2 \rfloor}. \label{eq-sub2:step1-sqf-1-odd}
\end{align}
\end{theorem}

Theorem~\ref{thm-sub2:main} follows from Theorem~\ref{thm-sub2:step1} and the following general formula.

\begin{theorem}\label{thm-sub2:step2}
Suppose that $p$ is odd.
Fix some $a_1, a_2 \in \F_p$, and write $\gamma = a_2 y^2 + a_1 y + 1 \in G_p$.
Let $m_1, m_2, \ldots, m_k$ be non-negative integers, and suppose that exactly $t \leq k$ of them are divisible by $p$.
Denote $M_{\sigma} = \sum_j m_j$ and $M_{\pi} = \prod_{j, p \nmid m_j} m_j$; the product runs over indices $j$ with $p \nmid m_j$.
Then \[ \sum_{\substack{\chi \in \widehat{G_p} \\ \chi(y^2 + 1) \neq 1}} \chi(\gamma)^{-1} \prod_{j = 1}^k C_{\chi^{m_j}} =
    \begin{cases}
        0 & \text{if } p \mid M_{\sigma} \text{ and } a_1 \neq 0, \\
        \left(\dfrac{M_{\pi}}{p}\right) p^t \delta_{k - t + 1, p}(2 a_2) & \text{if } p \mid M_{\sigma} \text{ and } a_1 = 0, \\
        \left(\dfrac{M_{\sigma} M_{\pi}}{p}\right) p^t \delta_{k - t, p}(2 a_2 - (1 - 1/\overline{M_{\sigma}}) a_1^2) & \text{if } p \nmid M_{\sigma},
    \end{cases} \]
    where $\delta_{n, p}$ is as defined in~\eqref{eq-sub2:delta}.
\end{theorem}

We note that when applying Theorem~\ref{thm-sub2:step2} to the formulas in Theorem~\ref{thm-sub2:step1}, we always have $\gamma = \gamma_0$, $M_{\sigma} = n$, and $t \in \{0, 1\}$, with $t = 1$ only when reducing~\eqref{eq-sub2:step1-sqf-1-odd} on the summation terms with $p \mid n - k$.

\subsection{Proof of Theorem~\ref{thm-sub2:step1}}

We first work towards extracting the densities of $V_n(\Z_p)_{a_1 = b_1}$ in the density formula for $V_n(\Z_p)_{a_1 = b_1, a_2 = b_2}$.
Recall the monoid homomorphism $\varphi_1 : \F_p[x]_m \to \F_p$ defined in~\eqref{eq:varphi1-def} by the formula $\varphi_1(1) = 0$ and
\[ \varphi_1(x^n + a_1 x^{n - 1} + \ldots + a_n) = a_1. \]
Let $\tau : G_p \to \F_p$ be unique group homomorphism satisfying $\tau \circ \varphi_2 = \varphi_1$, defined explicitly by
\[ \tau(a_2 y^2 + a_1 y + 1) = a_1. \]
The main property of $\tau$ we will use is the following.

\begin{proposition}\label{prop-sub2:tau-prop}
The map $\tau^* : \widehat{\F_p} \to \widehat{G_p}$ defined by $\rho \mapsto \rho \circ \tau$ is injective, and its image is precisely the set of characters $\chi \in \widehat{G_p}$ such that $\chi(y^2 + 1) = 1$.
\end{proposition}
\begin{proof}
Since $\tau^*$ is the induced dual map of $\tau$, injectivity of $\tau^*$ follows from the fact that $\tau$ is surjective.
The image of $\tau^*$ is precisely the set of $\chi \in \widehat{G_p}$ such that $\ker(\tau) \subseteq \ker(\chi)$,
where $\ker$ denotes the kernel.
The proposition follows since $\ker(\tau) = \langle y^2 + 1 \rangle$.
\end{proof}

We now deduce the desired formulas for $P^{\max}(V_n(\Z_p)_{a_1 = b_1, a_2 = b_2})$ and $P_0^{\sqf}(V_n(\Z_p)_{a_1 = b_1, a_2 = b_2})$.
Using the formula~\ref{eq-sub2:weightFT} for $\hat{w}$, Theorem~\ref{thm-density:max} and Theorem~\ref{thm-density:sqf-0} implies
\begin{align}
    P^{\max}(V_n(\Z_p)_{a_1 = b_1, a_2 = b_2}) &= \frac{1}{p^n} \sum_{\chi \in \widehat{G_p}} \chi(\gamma_0)^{-1} \left[\frac{L_{\chi \circ \varphi_2}(T)}{L_{\chi^2 \circ \varphi_2}(T^2/p)}\right]_n, \label{eq-sub2:max-init} \\
    P_0^{\sqf}(V_n(\Z_p)_{a_1 = b_1, a_2 = b_2}) &= \frac{1}{p^n} \sum_{\chi \in \widehat{G_p}} \chi(\gamma_0)^{-1} \left[\frac{L_{\chi \circ \varphi_2}(T)}{L_{\chi^2 \circ \varphi_2}(T^2)}\right]_n. \label{eq-sub2:sqf-0-init}
\end{align}
For any $z \in \C$, by Proposition~\ref{prop-sub2:tau-prop} and the equality $\tau \circ \varphi_2 = \varphi_1$, we have
\[ \frac{1}{p^n} \sum_{\substack{\chi \in \widehat{G_p} \\ \chi(y^2 + 1) = 1}} \chi(\gamma_0)^{-1} \left[\frac{L_{\chi \circ \varphi_2}(T)}{L_{\chi^2 \circ \varphi_2}(zT^2)}\right]_n = \frac{1}{p^n} \sum_{\rho \in \widehat{\F_p}} \rho(\overline{b_1})^{-1} \left[\frac{L_{\rho \circ \varphi_1}(T)}{L_{\rho^2 \circ \varphi_1}(T^2/p)}\right]_n. \]
Combining this equality with~\eqref{eq-sub1:max-init} and~\eqref{eq-sub1:sqf-0-init}, the contribution to the sums in~\eqref{eq-sub2:max-init} and~\eqref{eq-sub2:sqf-0-init} coming from terms $\chi \in \widehat{G_p}$ with $\chi(y^2 + 1) = 1$ are equal to $P^{\max}(V_n(\Z_p)_{a_1 = b_1})$ and $P_0^{\sqf}(V_n(\Z_p)_{a_1 = b_1})$, respectively.
It remains to count the contribution coming from terms $\chi \in \widehat{G_p}$ with $\chi(y^2 + 1) \neq 1$.
By the following proposition,~\eqref{eq-sub2:max-init} implies~\eqref{eq-sub2:step1-max-two} for $p = 2$ and~\eqref{eq-sub2:step1-max-odd} for $p$ odd, while~\eqref{eq-sub2:sqf-0-init} implies~\eqref{eq-sub2:step1-sqf-0-two} for $p = 2$ and~\eqref{eq-sub2:step1-sqf-0-odd} for $p$ odd.

\begin{proposition}\label{prop-sub2:char-gen2}
For any integer $n \geq 0$, $\chi \in \widehat{G_p}$ with $\chi^2 \neq 1$, and $z \in \C$,
\[ \left[\frac{L_{\chi \circ \varphi_2}(T)}{L_{\chi^2 \circ \varphi_2}(zT^2)}\right]_n = (-z)^{\lfloor n/2 \rfloor} C_{\chi}^{n - 2 \lfloor n/2 \rfloor} C_{\chi^2}^{\lfloor n/2 \rfloor}. \]
If $p = 2$ and $n \geq 2$, then both sides are zero.
\end{proposition}
\begin{proof}
The equality is immediate from Lemma~\ref{lem-sub2:char-gen},
which implies $L_{\chi \circ \varphi_2}(T) = 1 + C_{\chi} T$ and $L_{\chi^2 \circ \varphi_2}(zT^2) = 1 + C_{\chi^2} zT^2$.
For the case $p = 2$, it suffices to show that $C_{\chi^2} = 0$.
By definition, we have $C_{\chi^2} = 1 + \chi(y + 1)^2$, so it suffices to prove that $\chi(y + 1)^2 = -1$.
Since $\# \widehat{G_2} = \# G_2 = 4$, it is clear that $\chi(y + 1)^4 = 1$, so $\chi(y + 1)^2 = \pm 1$.
However, $y + 1$ generates $G_2$, so $\chi(y + 1)^2 = 1$ would imply $\chi^2 = 1$; contradiction.
\end{proof}

It remains to prove the formula~\eqref{eq-sub2:step1-sqf-1-odd} for $p$ odd.
Using the formula~\ref{eq-sub2:weightFT} for $\hat{w}$, Theorem~\ref{thm-density:sqf-1} implies
\[ P_1^{\sqf}(V_n(\Z_p)_{a_1 = b_1, a_2 = b_2}) = \frac{1 - 1/p}{p^n} \sum_{\chi \in \widehat{G_p}} \chi(\gamma_0)^{-1} \left[\sum_{c \in \F_p} \frac{\chi(\varphi_2(x + c))^2}{1 + \chi(\varphi_2(x + c)) \, T} \cdot \frac{L_{\chi \circ \varphi_2}(T)}{L_{\chi^2 \circ \varphi_2}(T^2)}\right]_{n - 2}. \]
Applying Proposition~\ref{prop-sub2:tau-prop} and the equality $\tau \circ \varphi_2 = \varphi_1$, we get
\begin{align*}
    P_1^{\sqf}(V_n(\Z_p)_{a_1 = b_1, a_2 = b_2})
    &= \frac{1 - 1/p}{p^n} \sum_{\rho \in \widehat{\F_p}} \rho(\overline{b_1})^{-1} \left[\sum_{c \in \F_p} \frac{\rho(\varphi_1(x + c))^2}{1 + \rho(\varphi_1(x + c)) \, T} \cdot \frac{L_{\rho \circ \varphi_1}(T)}{L_{\rho^2 \circ \varphi_1}(T^2)}\right]_{n - 2} \\
    &\qquad+ \frac{1 - 1/p}{p^n} \sum_{\substack{\chi \in \widehat{G_p} \\ \chi(y^2 + 1) \neq 1}} \chi(\gamma_0)^{-1} \left[\sum_{c \in \F_p} \frac{\chi(\varphi_2(x + c))^2}{1 + \chi(\varphi_2(x + c)) \, T} \cdot \frac{L_{\chi \circ \varphi_2}(T)}{L_{\chi^2 \circ \varphi_2}(T^2)}\right]_{n - 2}.
\end{align*}
The first term is equal to $P_1^{\sqf}(V_n(\Z_p)_{a_1 = b_1})$ by~\eqref{eq-sub1:sqf-1-init}, so
\begin{align*}
    & P_1^{\sqf}(V_n(\Z_p)_{a_1 = b_1, a_2 = b_2}) - P_1^{\sqf}(V_n(\Z_p)_{a_1 = b_1}) \\
    =\;& \frac{1 - 1/p}{p^n} \sum_{\substack{\chi \in \widehat{G_p} \\ \chi(y^2 + 1) \neq 1}} \chi(\gamma_0)^{-1} \left[\sum_{c \in \F_p} \frac{\chi(cy + 1)^2}{1 + \chi(cy + 1) \, T} \cdot \frac{L_{\chi \circ \varphi_2}(T)}{L_{\chi^2 \circ \varphi_2}(T^2)}\right]_{n - 2} \\
    =\;& \frac{1 - 1/p}{p^n} \sum_{\substack{\chi \in \widehat{G_p} \\ \chi(y^2 + 1) \neq 1}} \chi(\gamma_0)^{-1} \sum_{k = 0}^{n - 2} \left[\sum_{c \in \F_p} \frac{\chi(cy + 1)^2}{1 + \chi(cy + 1) \, T}\right]_{n - k - 2} \left[\frac{L_{\chi \circ \varphi_2}(T)}{L_{\chi^2 \circ \varphi_2}(T^2)}\right]_k \\
    =\;& \frac{1 - 1/p}{p^n} \sum_{\substack{\chi \in \widehat{G_p} \\ \chi(y^2 + 1) \neq 1}} \chi(\gamma_0)^{-1} \sum_{k = 0}^{n - 2} (-1)^{n - k - 2} \sum_{c \in \F_p} \chi(cy + 1)^{n - k} \left[\frac{L_{\chi \circ \varphi_2}(T)}{L_{\chi^2 \circ \varphi_2}(T^2)}\right]_k \\
    =\;& \frac{(-1)^n (1 - 1/p)}{p^n} \sum_{\substack{\chi \in \widehat{G_p} \\ \chi(y^2 + 1) \neq 1}} \sum_{k = 0}^{n - 2} (-1)^k \chi(\gamma_0)^{-1} C_{\chi^{n - k}} \left[\frac{L_{\chi \circ \varphi_2}(T)}{L_{\chi^2 \circ \varphi_2}(T^2)}\right]_k.
\end{align*}
Finally, by Proposition~\ref{prop-sub2:char-gen2}, for each $k \geq 0$, we have
\[ (-1)^k \left[\frac{L_{\chi \circ \varphi_2}(T)}{L_{\chi^2 \circ \varphi_2}(T^2)}\right]_k = (-1)^{k + \lfloor k/2 \rfloor} C_{\chi}^{k - 2 \lfloor k/2 \rfloor} C_{\chi^2}^{\lfloor k/2 \rfloor} = (-1)^{\lceil k/2 \rceil} C_{\chi}^{k - 2 \lfloor k/2 \rfloor} C_{\chi^2}^{\lfloor k/2 \rfloor}. \]
Subsituting into the previous equation yields the desired equality~\eqref{eq-sub2:step1-sqf-1-odd}.

\subsection{Proof of Theorem~\ref{thm-sub2:step2}}

Since removing the integers $m_j$ that are divisible by $p$ do not change $M_{\pi}$ and $M_{\sigma}$ modulo $p$ and since $C_1 = p$, we may assume without loss of generality that $t = 0$; that is, each $m_j$ is not divisible by $p$.
Define the function $e_p : \F_p \to \C$ by the formula
\[ e_p(c) = e^{2 \pi i \tilde{c}/p}, \]
    where $\tilde{c}$ is an arbitrary integer lift of $c$.
For any $b \in \F_p$, we denote
\[ \cG(p; b) = \sum_{c \in \F_p} e_p(bc^2). \]
This is the quadratic Gauss sum, and the exact formula for $\cG(p; b)$ is well-known; for example, see~\cite[Section 3.5]{IK2004}.
However, we will only need the formulas
\[ \cG(p; b) = \left(\frac{b}{p}\right) \cG(p; 1), \qquad \cG(p; b)^2 = \left(\frac{-1}{p}\right) p, \qquad \sum_{c \in \F_p} \left(\frac{c}{p}\right) e_p(cb) = \left(\frac{b}{p}\right) \cG(p; 1), \]
    which hold as long as $b \neq 0$; the last equality holds even if $b = 0$.
We first prove the following small but useful formula.

\begin{lemma}\label{lem-sub2:Gauss}
For any $\alpha, \beta \in \F_p$ with $\alpha \neq 0$,
\[ \sum_{c \in \F_p} e_p(\alpha c^2 + \beta c) = e_p\left(-\frac{\beta^2}{4 \alpha}\right) \cG(p; \alpha). \]
\end{lemma}
\begin{proof}
Indeed, by completing the squares,
\[ \sum_{c \in \F_p} e_p(\alpha c^2 + \beta c) = e_p\left(-\frac{\beta^2}{4 \alpha}\right) \sum_{c \in \F_p} e_p\left(\alpha \left(c + \frac{\beta}{2 \alpha}\right)^2\right) = e_p\left(-\frac{\beta^2}{4 \alpha}\right) \sum_{c \in \F_p} e_p(\alpha c^2), \]
    and the lemma follows by the definition of $\cG(p; \alpha)$.
\end{proof}

As an implication, we obtain an explicit formula for $C_{\chi}$ when $\chi(y^2 + 1) \neq 0$.

\begin{lemma}\label{lem-sub2:C-chi-Gauss}
Fix some $\chi \in \widehat{G_p}$ such that $\chi(y^2 + 1) \neq 1$.
Write $\chi(y^2/2 + y + 1) = e_p(c_1)$ and $\chi(y^2 + 1) = e_p(c_2)$ for some $c_1, c_2 \in \F_p$ such that $c_2 \neq 0$.
Then \[ C_{\chi} = e_p\left(\frac{c_1^2}{2 c_2}\right) \cG(p; -c_2/2). \]
\end{lemma}
\begin{proof}
Note that for any $c \in \F_p$, we have $cy + 1 = (y^2 + 1)^{-\tilde{c}^2/2} (y^2/2 + y + 1)^{\tilde{c}}$.
Thus, by expanding, we get
\[ C_{\chi} = \sum_{c \in \F_p} \chi(cy + 1) = \sum_{c \in \F_p} e_p\left(-\frac{c_2}{2} c^2 + c_1 c\right). \]
The lemma now follows from Lemma~\ref{lem-sub2:Gauss}.
\end{proof}

We now prove Theorem~\ref{thm-sub2:step2}.
First note that each $\chi \in \widehat{G_p}$ is uniquely determined by its value at $y^2/2 + y + 1$ and $y^2 + 1$.
By Lemma~\ref{lem-sub2:C-chi-Gauss}, we have
\begin{align*}
    \sum_{\substack{\chi \in \widehat{G_p} \\ \chi(y^2 + 1) \neq 1}} \chi(\gamma)^{-1} \prod_{j = 1}^k C_{\chi^{m_j}}
    &= \sum_{\substack{c_1, c_2 \in \F_p \\ c_2 \neq 0}} e_p\left(-c_1 a_1 - c_2 \left(a_2 - \frac{a_1^2}{2}\right)\right) \prod_{j = 1}^k e_p\left(\frac{c_1^2 m_j}{2 c_2}\right) \cG(p; -m_j c_2/2) \\
    &= \sum_{\substack{c_1, c_2 \in \F_p \\ c_2 \neq 0}} e_p\left(\frac{c_1^2 M_{\sigma}}{2 c_2} -c_1 a_1 - c_2 \left(a_2 - \frac{a_1^2}{2}\right)\right) \prod_{j = 1}^k \cG(p; -m_j c_2/2).
\end{align*}
By applying the formula $\cG(p; bc) = \left(\frac{b}{p}\right) \cG(p; c)$, we get
\begin{equation}\label{eq-sub2:step2-init}
    \sum_{\substack{\chi \in \widehat{G_p} \\ \chi(y^2 + 1) \neq 1}} \chi(\gamma)^{-1} \prod_{j = 1}^k C_{\chi^{m_j}} = \left(\frac{M_{\pi}}{p}\right) \sum_{\substack{c_1, c_2 \in \F_p \\ c_2 \neq 0}} e_p\left(\frac{c_1^2 M_{\sigma}}{2 c_2} -c_1 a_1 - c_2 \left(a_2 - \frac{a_1^2}{2}\right)\right) \cG(p; -c_2/2)^k.
\end{equation}

If $p \mid M_{\sigma}$ and $a_1 \neq 0$, then summing just over $c_1$ yields that the right hand side is zero.
If $p \mid M_{\sigma}$ and $a_1 = 0$, then~\eqref{eq-sub2:step2-init} yields
\begin{align*}
    \sum_{\substack{\chi \in \widehat{G_p} \\ \chi(y^2 + 1) \neq 1}} \chi(\gamma)^{-1} \prod_{j = 1}^k C_{\chi^{m_j}} 
    &= \left(\frac{M_{\pi}}{p}\right) \sum_{\substack{c_1, c_2 \in \F_p \\ c_2 \neq 0}} e_p(-c_2 a_2) \, \cG(p; -c_2/2)^k \\
    &= \left(\frac{M_{\pi}}{p}\right) \sum_{c_2 \in \F_p \setminus \{0\}} e_p(-c_2 a_2) \, p \cG(p; -c_2/2)^k.
\end{align*}
Then the formula given in Theorem~\ref{thm-sub2:step2} follows from the following proposition.

\begin{proposition}\label{prop-sub2:Missingno}
For any $a, b \in \F_p$ with $b \neq 0$ and for any integer $n$, we have
\[ \sum_{c \in \F_p \setminus \{0\}} e_p(ca) \, p \cG(p; cb)^n = \delta_{n + 1, p}(ab), \]
    where $\delta_{n, p}$ is as defined in~\eqref{eq-sub2:delta}.
\end{proposition}
\begin{proof}
If $n$ is even, then
\[ \sum_{c \in \F_p \setminus \{0\}} e_p(ca) \, p \cG(p; cb)^n = \sum_{c \in \F_p \setminus \{0\}} e_p(ca) \, \left(\frac{-1}{p}\right)^{n/2} p^{n/2 + 1}, \]
    which is equal to $\delta_{n + 1, p}(ab)$ since $n + 1$ is odd and $ab = 0$ if and only if $a = 0$.
If $n$ is odd, then we have
\[ \sum_{c \in \F_p \setminus \{0\}} e_p(ca) \, p \cG(p; cb)^n = \sum_{c \in \F_p \setminus \{0\}} e_p(ca) \left(\frac{cb}{p}\right) \cdot p \cG(p; 1)^n = \sum_{c \in \F_p \setminus \{0\}} e_p(cab) \left(\frac{c}{p}\right) \cdot p \cG(p; 1)^n. \]
The sum on the right hand side is also known as a quadratic Gauss sum, with value $\left(\frac{ab}{p}\right) \cG(p; 1)$, so
\[ \sum_{c \in \F_p \setminus \{0\}} e_p(ca) \, p \cG(p; cb)^n = \left(\frac{ab}{p}\right) \, p \cG(p; 1)^{n + 1} = \left(\frac{ab}{p}\right) \left(\frac{-1}{p}\right)^{(n + 1)/2} p^{(n + 1)/2 + 1}. \]
The right hand side is equal to $\delta_{n + 1, p}(ab)$, since $n + 1$ is even.
\end{proof}

It remains to solve the case $p \nmid M_{\sigma}$.
Applying Lemma~\ref{lem-sub2:C-chi-Gauss} to~\eqref{eq-sub2:step2-init} yields
\begin{align*}
    \sum_{\substack{\chi \in \widehat{G_p} \\ \chi(y^2 + 1) \neq 1}} \chi(\gamma)^{-1} \prod_{j = 1}^k C_{\chi^{m_j}}
    &= \left(\frac{M_{\pi}}{p}\right) \sum_{c_2 \in \F_p \setminus \{0\}} e_p\left(\frac{c_1^2 M_{\sigma}}{2 c_2} -c_1 a_1 - c_2 \left(a_2 - \frac{a_1^2}{2}\right)\right) \cG(p; -c_2/2)^k \\
    &= \left(\frac{M_{\pi}}{p}\right) \sum_{c_2 \in \F_p \setminus \{0\}} e_p\left(-\frac{2 c_2 a_1^2}{4 \overline{M_{\sigma}}} - c_2 \left(a_2 - \frac{a_1^2}{2}\right)\right) \cG(p; \overline{M_{\sigma}}/(2 c_2)) \cG(p; -c_2/2)^k \\
    &= \left(\frac{M_{\pi} M_{\sigma}}{p}\right) \sum_{c_2 \in \F_p \setminus \{0\}} e_p\left(-\frac{c_2}{2} (2a_2 - (1 - 1/\overline{M_{\sigma}}) a_1^2)\right) \, \left(\frac{-1}{p}\right) \cG(p; -c_2/2)^{k + 1} \\
    &= \left(\frac{M_{\pi} M_{\sigma}}{p}\right) \sum_{c_2 \in \F_p \setminus \{0\}} e_p\left(-\frac{c_2}{2} (2a_2 - (1 - 1/\overline{M_{\sigma}}) a_1^2)\right) \, p \cG(p; -c_2/2)^{k - 1}.
\end{align*}
By Proposition~\ref{prop-sub2:Missingno}, the right hand side is equal to $\left(\frac{M_{\pi} M_{\sigma}}{p}\right) \delta_{k, p}(2 a_2 - (1 - 1/\overline{M_{\sigma}}) a_1^2)$.
We are done.

\section{Density formula for polynomials with unit constant coefficient}\label{section:density-formula-unit}

Fix an integer $n \geq 2$ and a prime number $p$.
Recall that $V_n(\Z_p)$ is the set of degree $n$ monic polynomials over $\Z_p$ and $\lambda_p$ is the $p$-adic Haar measure on $V_n(\Z_p)$.
Let $\Sigma$ be a subset of $V_n(\Z_p)$ defined by congruence conditions mod $p^2$.
Recall that $P_0^{\sqf}(\Sigma)$ and $P_1^{\sqf}(\Sigma)$ denote the density of polynomials in $\Sigma$ whose discriminant has $p$-adic valuation $0$ and $1$, respectively, and $P^{\max}(\Sigma)$ denote the density of polynomials $f \in \Sigma$ such that $\Z_p[x]/(f(x))$ is the maximal order of $\Q_p[x]/(f(x))$.
In this section, we state and prove formulas for $P_0^{\sqf}(\Sigma)$, $P_1^{\sqf}(\Sigma)$, and $P^{\max}(\Sigma)$ under the assumption that $f(0) \in \Z_p^{\times}$ for every $f \in \Sigma$.
To do this, we modify the proof of the density formulas proved in \S\ref{section:density-formula}.

Let $\F_p[x]_{m, x \nmid u}$ denote the set of polynomials $u \in \F_p[x]_m$ such that $x \nmid u$.
For any non-negative integer $k < n$, recall that the $x^k$-coefficient of $\Sigma$ is \emph{defined by congruence conditions mod $p$} if for any $f \in \Sigma$ and $c \in \Z_p$ such that $p \mid c$, we have $f + cx^k \in \Sigma$.
Next, we need to modify the definition of $\Sigma$-admissible triples given in Definition~\ref{def:admissible}.

\begin{definition}\label{def:unit-admissible}
A \emph{$\Sigma$-unit-admissible triple} is a triple $(G, \psi, w)$, where $G$ is a finite abelian group, $\psi : \F_p[x]_{m, x \nmid u} \to G$ is a monoid homomorphism, and $w : G \to \C$ is a function such that for any $u \in \F_p[x]_{m, x \nmid u}$ of degree $n$,
\[ w(\psi(u)) = \frac{\lambda_p(\{f \in \Sigma : \overline{f} = u\})}{\lambda_p(\Sigma)}. \]
\end{definition}

The general formulas to be proved in this section are as follows.

\begin{theorem}\label{thm-density-unit:sqf-0}
Suppose that $f(0) \in \Z_p^{\times}$ for every $f \in \Sigma$.
Let $(G, \psi, w)$ be a $\Sigma$-unit-admissible triple.
Then \[ P_0^{\sqf}(\Sigma) = \sum_{\chi \in \widehat{G}} \hat{w}(\chi) \left[\frac{L_{\chi \circ \psi}(T)}{L_{\chi^2 \circ \psi}(T^2)}\right]_n. \]
\end{theorem}

\begin{theorem}\label{thm-density-unit:sqf-1}
Suppose that $f(0) \in \Z_p^{\times}$ for every $f \in \Sigma$, and suppose that there exists $k < n$ such that the $x^k$-coefficient of $\Sigma$ is defined by congruence conditions mod $p$.
Let $(G, \psi, w)$ be a $\Sigma$-unit-admissible triple.
If $p = 2$, then $P_1^{\sqf}(\Sigma) = 0$, and if $p$ is odd, then
\[ P_1^{\sqf}(\Sigma) = \left(1 - \frac{1}{p}\right) \sum_{\chi \in \widehat{G}} \hat{w}(\chi) \left[\sum_{c \in \F_p^{\times}} \frac{\chi(\psi(x + c))^2}{1 + \chi(\psi(x + c)) \, T} \cdot \frac{L_{\chi \circ \psi}(T)}{L_{\chi^2 \circ \psi}(T^2)}\right]_{n - 2}. \]
\end{theorem}

\begin{theorem}\label{thm-density-unit:max}
Suppose that $f(0) \in \Z_p^{\times}$ for every $f \in \Sigma$, and suppose that there exists a non-negative integer $k_0 \leq \lceil n/2 \rceil$ such that the $x^k$-coefficient of $\Sigma$ is defined by congruence conditions mod $p$ for all integers $k$ with $k_0 \leq k < k_0 + \lfloor n/2 \rfloor$.
Let $(G, \psi, w)$ be a $\Sigma$-unit-admissible triple.
Then \[ P^{\max}(\Sigma) = \sum_{\chi \in \widehat{G}} \hat{w}(\chi) \left[\frac{L_{\chi \circ \psi}(T)}{L_{\chi^2 \circ \psi}(T^2/p)}\right]_n. \]
\end{theorem}

Theorem~\ref{thm-density-unit:sqf-0} follows from the proof of Theorem~\ref{thm-density:sqf-0} using Lemma~\ref{lem:sqfree-criterion}, replacing $\F_p[x]_m$ with $\F_p[x]_{m, x \nmid u}$.
Similarly, Theorem~\ref{thm-density-unit:sqf-1} follows from the proof of Theorem~\ref{thm-density:sqf-1} using Lemma~\ref{lem:sqfree-criterion}, with Proposition~\ref{prop-density:sqf-1-lift-count} replaced by the following analogue.

\begin{proposition}
Suppose that there exists $k < n$ such that the $x^k$-coefficient of $\Sigma$ is defined by congruence conditions mod $p$.
Then for any $c \in \F_p^{\times}$ and $u \in \F_p[x]_{m, x \nmid u}$ of degree $n - 2$, letting $\tilde{c}$ be an arbitrary $p$-adic lift of $c$, we have
\[ \lambda_p(\{f \in \Sigma : \overline{f} = (x + c)^2 u, p^2 \nmid f(-\tilde{c})\}) = \left(1 - \frac{1}{p}\right) \lambda_p(\{f \in \Sigma : \overline{f} = (x + c)^2 u\}). \]
\end{proposition}
\begin{proof}
The proof is similar to those of Proposition~\ref{prop-density:sqf-1-lift-count}.
Fix all coefficients of $f$ except the $x^k$-coefficient.
The $x^k$-coefficient of $\overline{f}$ has $p$ possible lifts mod $p^2$, and $p^2 \mid f(-\tilde{c})$ holds for exactly one of them since $c \neq 0$.
\end{proof}

Finally, we prove Theorem~\ref{thm-density-unit:max} by modifying the proof of Theorem~\ref{thm-density:max}.
Lemma~\ref{lem-density:max-overgeneral} still holds with $\F_p[x]_m$ replaced by $\F_p[x]_{m, x \nmid u}$, so Theorem~\ref{thm-density-unit:max} reduces to the following analogue of Proposition~\ref{prop-density:max-reduce}.

\begin{proposition}
Suppose that there exists $k_0 \leq \lceil n/2 \rceil$ such that the $x^k$-coefficient of $\Sigma$ is defined by congruence conditions mod $p$ for all integers $k$ with $k_0 \leq k < k_0 + \lfloor n/2 \rfloor$.
Then for any $u \in \F_p[x]_{m, x \nmid u}$,
\[ \lambda_p(\Sigma \cap (p, \tilde{u})^2) = p^{-\deg(u)} \, \lambda_p(\{f \in \Sigma : u^2 \mid \overline{f}\}). \]
\end{proposition}
\begin{proof}
We modify the proof of Proposition~\ref{prop-density:max-reduce}.
We denote $d = \deg(u)$; we may assume that $d \leq \lfloor n/2 \rfloor$.
We redefine the map $\phi : \left(\Sigma_{\bmod p^2} \cap (p, \tilde{u})^2\right) \times A_d \to \{f \in \Sigma_{\bmod{p^2}} : u^2 \mid \overline{f}\}$ by the formula
\[ \phi(f, g) = f + x^{k_0} g. \]
Note that $\phi$ is still well-defined since $k_0 + d \leq k_0 + \lfloor n/2 \rfloor \leq n$.
As in the proof of Proposition~\ref{prop-density:max-reduce}, it suffices to show that $\phi$ is bijective.
The proof of injectivity works as is, with $g_1$ and $g_2$ replaced by $x^{k_0} g_1$ and $x^{k_0} g_2$, since $x \nmid u$ means that $x^{k_0} g_2 - x^{k_0} g_1 \in (p \tilde{u})$ implies $g_2 - g_1 \in (p \tilde{u})$.
For the proof of surjectivity, we write $g_0 = ug_1 + x^{k_0} g_2$ instead of $g_0 = ug_1 + g_2$; the desired polynomials $g_1, g_2 \in \F_p[x]_m$ with $\deg(g_2) < d$ still exist since $u$ and $x^{k_0}$ are coprime.
Thus, as opposed to $h = \tilde{u}^2 f_0 + p \tilde{u} \tilde{g_1} + p \tilde{g_2}$, we have
\[ h = \tilde{u}^2 f_0 + p \tilde{u} \tilde{g_1} + p x^{k_0} \tilde{g_2} \]
    with $\tilde{u}^2 f_0 + p \tilde{u} \tilde{g_1} \in \Sigma_{\bmod{p^2}} \cap (p, \tilde{u})^2$ and $p \tilde{g_2} \in A_d$.
This proves that $\phi$ is surjective.
\end{proof}

\section{Proof of Theorem~\ref{thm-local:copconst}}\label{section:copconst}

Let $n \geq 2$ and $p$ be a prime number.
Recall that
\[ V_n(\Z_p)_{p \nmid a_n} = \{f(x) = x^n + a_1 x^{n - 1} + \ldots + a_n \in \Z_p[x] : p \nmid a_n\}. \]
Theorem~\ref{thm-local:copconst} follows from the following exact formula for the densities of $V_n(\Z_p)_{p \nmid a_n}$.

\begin{theorem}\label{thm-copconst:main}
For any $n \geq 2$ and prime number $p$,
\begin{align}
    P^{\max}(V_n(\Z_p)_{p \nmid a_n}) &= 1 - \frac{1}{p^2 + p + 1} + \begin{cases} \dfrac{1}{(p^2 + p + 1) p^{3(n - 1)/2}} & \text{if } 2 \nmid n, \\ -\dfrac{p + 1}{(p^2 + p + 1) p^{3n/2 - 1}} & \text{if } 2 \mid n, \end{cases} \label{eq-copconst:max} \\
    P_0^{\sqf}(V_n(\Z_p)_{p \nmid a_n}) &= \frac{p}{p + 1} \left(1 - \frac{(-1)^n}{p^n}\right). \label{eq-copconst:sqf-0}
\end{align}
We have $P_1^{\sqf}(V_n(\Z_2)_{2 \nmid a_n}) = 0$ and if $p$ is odd, then
\begin{equation}\label{eq-copconst:sqf-1}
    P_1^{\sqf}(V_n(\Z_p)_{p \nmid a_n}) = \frac{(p - 1)^2}{p(p + 1)^2} - \frac{(-1)^n}{p^n} \left((n - 1) \frac{(p - 1)^2}{p + 1} - \frac{p(p - 1)(p + 3)}{(p + 1)^2}\right).
\end{equation}
\end{theorem}

It remains to prove Theorem~\ref{thm-copconst:main}.
First note that $V_n(\Z_p)_{p \nmid a_n}$ is defined by congruence conditions mod $p$, and $f(0) \in \Z_p^{\times}$ for any $f \in V_n(\Z_p)_{p \nmid a_n}$.
Thus, the formulas from \S\ref{section:density-formula-unit} apply directly to $V_n(\Z_p)_{p \nmid a_n}$.
We start by finding a $V_n(\Z_p)_{p \nmid a_n}$-unit-admissible triple $(G, \psi, w)$ as defined in Definition~\ref{def:unit-admissible} and computing $\hat{w}(\chi)$ and $L_{\chi \circ \psi}(T)$ for all $\chi \in \widehat{G}$.
Then we apply the formulas obtained in \S\ref{section:density-formula-unit} to prove Theorem~\ref{thm-copconst:main}.

It is easy to see that for any $u \in \F_p[x]_{m, x \nmid u}$ of degree $n$, we have
\[ \frac{\lambda_p(\{f \in V_n(\Z_p)_{p \nmid a_n} : \overline{f} = u\})}{\lambda_p(V_n(\Z_p)_{p \nmid a_n})} = \frac{1}{(p - 1) p^{n - 1}}. \]
Thus we can take $G = \{1\}$, the trivial group, $\psi : \F_p[x]_{m, x \nmid u} \to \{1\}$ the trivial map, and the weight $w(1) = 1/((p - 1) p^{n - 1})$, giving us a $V_n(\Z_p)_{p \nmid a_n}$-unit-admissible triple $(G, \psi, w)$.
Clearly, we have
\begin{equation}\label{eq-copconst:weightFT}
    \hat{w}(1) = w(1) = \frac{1}{(p - 1) p^{n - 1}}.
\end{equation}
Let $\mathbf{1}_x : \F_p[x]_{m, x \nmid u} \to \C$ denote the all-one function on $\F_p[x]_{m, x \nmid u}$.
Then $1 \circ \psi = \mathbf{1}_x$, where here $1$ denotes the trivial character of $\{1\}$.
Direct computation shows that
\begin{equation}\label{eq:1x-gen}
    L_{\mathbf{1}_x}(T) = \frac{1 - T}{1 - pT}.
\end{equation}

We are now ready to prove Theorem~\ref{thm-copconst:main}.
We first compute $P^{\max}(V_n(\Z_p)_{p \nmid a_n})$.
Using the formula~\eqref{eq-copconst:weightFT} for $\hat{w}$, Theorem~\ref{thm-density-unit:max} yields
\begin{equation}\label{eq-copconst:max-init}
    P^{\max}(V_n(\Z_p)_{p \nmid a_n}) = \frac{1}{(p - 1) p^{n - 1}} \left[\frac{L_{\mathbf{1}_x}(T)}{L_{\mathbf{1}_x}(T^2/p)}\right]_n.
\end{equation}
By~\eqref{eq:1x-gen}, we have
\[ \frac{L_{\mathbf{1}_x}(T)}{L_{\mathbf{1}_x}(T^2/p)} = \frac{(1 - T)(1 - T^2)}{(1 - pT)(1 - T^2/p)} = 1 + \frac{(p - 1)(p + 1)}{p^2 + p + 1} \cdot \frac{1}{1 - pT} - \frac{(p - 1)(p + 1 - T)}{p^2 + p + 1} \cdot \frac{1}{1 - T^2/p}. \]
As a result, since $n \geq 2$, we get
\begin{align*}
    P^{\max}(V_n(\Z_p)_{p \nmid a_n})
    &= \frac{1}{(p - 1) p^{n - 1}} \left(\frac{(p - 1)(p + 1)}{p^2 + p + 1} \cdot p^n - \frac{(p - 1) [p + 1 - T]_{n - 2 \lfloor n/2 \rfloor}}{p^2 + p + 1} \cdot \frac{1}{p^{\lfloor n/2 \rfloor}}\right) \\
    &= \frac{p^2 + p}{p^2 + p + 1} - \frac{[p + 1 - T]_{n - 2 \lfloor n/2 \rfloor}}{(p^2 + p + 1) p^{n - 1 + \lfloor n/2 \rfloor}}.
\end{align*}
Since $n - 2 \lfloor n/2 \rfloor$ is equal to $1$ if $n$ is odd and $0$ if $n$ is even, this equality implies~\eqref{eq-copconst:max}.

Next, we compute $P_0^{\sqf}(V_n(\Z_p)_{p \nmid a_n})$.
Using the formula~\eqref{eq-copconst:weightFT} for $\hat{w}$, Theorem~\ref{thm-density-unit:sqf-0} yields
\begin{equation}\label{eq-copconst:sqf-0-init}
    P_0^{\sqf}(V_n(\Z_p)_{p \nmid a_n}) = \frac{1}{(p - 1) p^{n - 1}} \left[\frac{L_{\mathbf{1}_x}(T)}{L_{\mathbf{1}_x}(T^2)}\right]_n.
\end{equation}
By~\eqref{eq:1x-gen}, we have
\[ \frac{L_{\mathbf{1}_x}(T)}{L_{\mathbf{1}_x}(T^2)} = \frac{1 - pT^2}{(1 + T)(1 - pT)} = 1 + \frac{p - 1}{p + 1} \left(\frac{1}{1 - pT} - \frac{1}{1 + T}\right). \]
As a result, since $n \geq 2$, we get
\[ P_0^{\sqf}(V_n(\Z_p)_{p \nmid a_n}) = \frac{1}{(p - 1) p^{n - 1}} \cdot \frac{p - 1}{p + 1} (p^n - (-1)^n) = \frac{p}{p + 1} \left(1 - \frac{(-1)^n}{p^n}\right). \]
This proves~\eqref{eq-copconst:sqf-0}.

Finally, we compute $P_1^{\sqf}(V_n(\Z_p)_{p \nmid a_n})$ for $p$ odd.
Using the formula~\eqref{eq-copconst:weightFT} for $\hat{w}$, Theorem~\ref{thm-density-unit:sqf-1} yields
\begin{equation}\label{eq-copconst:sqf-1-init}
    P_1^{\sqf}(V_n(\Z_p)_{p \nmid a_n}) = \frac{1 - 1/p}{(p - 1) p^{n - 1}} \left[\frac{p - 1}{1 + T} \cdot \frac{L_{\mathbf{1}_x}(T)}{L_{\mathbf{1}_x}(T^2)}\right]_{n - 2} = \frac{p - 1}{p^n} \left[\frac{1}{1 + T} \cdot \frac{L_{\mathbf{1}_x}(T)}{L_{\mathbf{1}_x}(T^2)}\right]_{n - 2}.
\end{equation}
By~\eqref{eq:1x-gen}, we have
\[ \frac{1}{1 + T} \cdot \frac{L_{\mathbf{1}_x}(T)}{L_{\mathbf{1}_x}(T^2)} = \frac{1 - pT^2}{(1 + T)^2 (1 - pT)} = \frac{p(p - 1)}{(p + 1)^2} \cdot \frac{1}{1 - pT} - \frac{p - 1}{p + 1} \cdot \frac{1}{(1 + T)^2} + \frac{p(p + 3)}{(p + 1)^2} \cdot \frac{1}{1 + T}. \]
As a result, we get
\begin{align*}
    P_1^{\sqf}(V_n(\Z_p)_{p \nmid a_n})
    &= \frac{p - 1}{p^n} \left(\frac{p(p - 1)}{(p + 1)^2} \cdot p^{n - 2} - \frac{p - 1}{p + 1} \cdot (-1)^n (n - 1) + \frac{p(p + 3)}{(p + 1)^2} \cdot (-1)^n\right) \\
    &= \frac{(p - 1)^2}{p(p + 1)^2} - \frac{(-1)^n}{p^n} \left((n - 1) \frac{(p - 1)^2}{p + 1} - \frac{p(p - 1)(p + 3)}{(p + 1)^2}\right).
\end{align*}
This proves~\eqref{eq-copconst:sqf-1} and finishes the proof of Theorem~\ref{thm-copconst:main}.

\section{Proof of Theorem~\ref{thm-local:fixconst}}\label{section:fixconst}

Let $n \geq 2$ and $p$ be a prime number.
Fix $b_n \in \Z_p^{\times}$, and recall that
\begin{align*}
    V_n(\Z_p)_{a_n = b_n} &= \{f(x) = x^n + a_1 x^{n - 1} + \ldots + a_n \in \Z_p[x] : a_n = b_n\}, \\
    V_n(\Z_p)_{p \nmid a_n} &= \{f(x) = x^n + a_1 x^{n - 1} + \ldots + a_n \in \Z_p[x] : p \nmid a_n\}.
\end{align*}
We first state the exact formula for the densities of $V_n(\Z_p)_{a_n = b_n}$ in terms of the densities of $V_n(\Z_p)_{p \nmid a_n}$.
For each $n \geq 2$, odd prime number $p$, and $b \in \F_p^{\times}$, define the quantity
\begin{equation}\label{eq-fixconst:kappa}
    \kappa_{n, p}(b) = \frac{1}{p - 1} \sum_{\chi \in \widehat{\F_p^{\times}} \setminus \{1\}} \chi(\overline{b_n})^{-1} \sum_{c \in \F_p^{\times}} \chi(c)^n = \begin{cases} \gcd(n, p - 1) - 1 & \text{if } b \in \F_p^{\times} \text{ is an $n$th power,} \\ -1 & \text{if } b \in \F_p^{\times} \text{ is not an $n$th power.} \end{cases}
\end{equation}
Note that the second equality follows by applying character sum formula on $\widehat{\F_p^{\times}}$ and orthogonality of characters on the subgroup $\{\chi \in \widehat{\F_p^{\times}} : \chi^n = 1\}$, which is cyclic of size $\gcd(n, p - 1)$.
Then we prove:

\begin{theorem}\label{thm-fixconst:main}
If $p = 2$, then $P_1^{\sqf}(V_n(\Z_2)_{a_n = b_n}) = 0$ and
\begin{align}
    P^{\max}(V_n(\Z_2)_{a_n = b_n}) &= P^{\max}(V_n(\Z_2)_{2 \nmid a_n}), \label{eq-fixconst:max-two} \\
    P_0^{\sqf}(V_n(\Z_2)_{a_n = b_n}) &= P_0^{\sqf}(V_n(\Z_2)_{2 \nmid a_n}). \label{eq-fixconst:sqf-0-two}
\end{align}

If $p$ is odd and $2 \nmid n$, then
\begin{align}
    P^{\max}(V_n(\Z_p)_{a_n = b_n}) &= P^{\max}(V_n(\Z_p)_{p \nmid a_n}), \label{eq-fixconst:max-odd-deg-odd} \\
    P_0^{\sqf}(V_n(\Z_p)_{a_n = b_n}) &= P_0^{\sqf}(V_n(\Z_p)_{p \nmid a_n}), \label{eq-fixconst:sqf-0-odd-deg-odd} \\
    P_1^{\sqf}(V_n(\Z_p)_{a_n = b_n}) &= P_1^{\sqf}(V_n(\Z_p)_{p \nmid a_n}) - \frac{p - 1}{p^n} \kappa_{n, p}(\overline{b_n}). \label{eq-fixconst:sqf-1-odd-deg-odd}
\end{align}

If $p$ is odd and $2 \mid n$, then
\begin{align}
    P^{\max}(V_n(\Z_p)_{a_n = b_n}) &= P^{\max}(V_n(\Z_p)_{p \nmid a_n}) - \frac{1}{p^{3n/2 - 1}} \left(\frac{b_n}{p}\right), \label{eq-fixconst:max-odd-deg-even} \\
    P_0^{\sqf}(V_n(\Z_p)_{a_n = b_n}) &= P_0^{\sqf}(V_n(\Z_p)_{p \nmid a_n}) - \frac{1}{p^{n - 1}} \left(\frac{b_n}{p}\right), \label{eq-fixconst:sqf-0-odd-deg-even} \\
    P_1^{\sqf}(V_n(\Z_p)_{a_n = b_n}) &= P_1^{\sqf}(V_n(\Z_p)_{p \nmid a_n}) + \frac{p - 1}{p^n} \kappa_{n, p}(\overline{b_n}) - \frac{(p - 1)^2}{p^n} \left(\frac{n}{2} - 1\right) \left(\frac{b_n}{p}\right). \label{eq-fixconst:sqf-1-odd-deg-even}
\end{align}
\end{theorem}

The asymptotic density formulas in Theorem~\ref{thm-local:fixconst} follows from Theorem~\ref{thm-fixconst:main}, Theorem~\ref{thm-local:copconst} on the asymptotic densities of $V_n(\Z_p)_{p \nmid a_n}$, and the estimate $\kappa_{n, p}(b) = O(n)$.
The densities $P_0^{\sqf}(V_n(\Z_p)_{a_n = b_n})$ and $P^{\max}(V_n(\Z_p)_{a_n = b_n})$ are positive since $x^n + b_n$ is squarefree mod $p$ if $p \nmid n$ while $x^n - x + b_n$ is squarefree mod $p$ if $p \mid n$.

It remains to prove Theorem~\ref{thm-fixconst:main}.
First, since squarefree discriminant and maximality are mod $p^2$ conditions, the densities for $V_n(\Z_p)_{a_n = b_n}$ are the same as those of the set
\[ V_n(\Z_p)_{a_n \equiv b_n \bmod{p^2}} = \{f(x) = x^n + a_1 x^{n - 1} + \ldots + a_n \in \Z_p[x] : a_n \equiv b_n \bmod{p^2}\}, \]
    which is defined by congruence conditions mod $p^2$.
We start by finding a $V_n(\Z_p)_{a_n \equiv b_n \bmod{p^2}}$-unit-admissible triple $(G, \psi, w)$ as defined in Definition~\ref{def:unit-admissible} and computing $\hat{w}(\chi)$ and $L_{\chi \circ \psi}(T)$ for all $\chi \in \widehat{G}$.
Then we apply the formulas obtained in \S\ref{section:density-formula-unit} to prove Theorem~\ref{thm-fixconst:main}.

It is easy to see that for any $u \in \F_p[x]_{m, x \nmid u}$ of degree $n$,
\[ \frac{\lambda_p(\{f \in V_n(\Z_p)_{a_n \equiv b_n \bmod{p^2}} : \overline{f} = u\})}{\lambda_p(V_n(\Z_p)_{a_n \equiv b_n \bmod{p^2}})} = \begin{cases} 1/p^{n - 1} & \text{if } u(0) = \overline{b_n}, \\ 0 & \text{if } u(0) \neq \overline{b_n}. \end{cases} \]
Define the monoid homomorphism $\ev_0 : \F_p[x]_{m, x \nmid u} \to \F_p^{\times}$ by the formula
\begin{equation}\label{eq-fixconst:ev0}
    \ev_0(u) = u(0).
\end{equation}
Therefore, $(\F_p^{\times}, \ev_0, w)$ is $V_n(\Z_p)_{a_n \equiv b_n \bmod{p^2}}$-unit-admissible, where the weight function $w : \F_p \to \C$ is defined by
\[ w(c) = \begin{cases} 1/p^{n - 1} & \text{if } c = \overline{b_n}, \\ 0 & \text{if } c \neq \overline{b_n}. \end{cases} \]
By direct computation, the Fourier transform $\hat{w}$ of $w$ has the formula
\begin{equation}\label{eq-fixconst:weightFT}
    \hat{w}(\chi) = \frac{1}{\# \F_p^{\times}} \sum_{c \in \F_p^{\times}} w(c) \chi(c)^{-1} = \frac{1}{(p - 1) p^{n - 1}} \chi(\overline{b_n})^{-1}.
\end{equation}

Now we compute $L_{\chi \circ \ev_0}(T)$ for any $\chi \in \widehat{\F_p^{\times}}$ such that $\chi \neq 1$.
Note that $1 \circ \ev_0 = \mathbf{1}_x$, where $\mathbf{1}_x : \F_p[x]_{m, x \nmid u} \to \C$ is the all-one function on $\F_p[x]_{m, x \nmid u}$.

\begin{lemma}\label{lem-fixconst:char-gen}
For any $\chi \in \widehat{\F_p^{\times}}$ such that $\chi \neq 1$, we have $L_{\chi \circ \ev_0}(T) = 1$.
\end{lemma}
\begin{proof}
Clearly $[L_{\chi \circ \ev_0}(T)]_0 = 1$.
Now fix some $n \geq 1$.
For each $c \in \F_p^{\times}$, there exists exactly $p^{n - 1}$ polynomials in $\F_p[x]_{m, x \nmid u}$ whose constant coefficient is $c$.
Thus, we get
\[ \left[L_{\chi \circ \ev_0}(T)\right]_n = \sum_{\substack{u \in \F_p[x]_{m, x \nmid u} \\ \deg(u) = n}} \chi(u(0)) = p^{n - 1} \sum_{c \in \F_p^{\times}} \chi(c), \]
    which is equal to zero since $\chi \neq 1$.
\end{proof}

If $p \neq 2$, then the quadratic character over $\F_p^{\times}$, which we denote by $\left(\frac{\cdot}{p}\right)$, is the unique $\chi \in \widehat{\F_p^{\times}}$ such that $\chi \neq 1$ and $\chi^2 = 1$.
Thus, combining with the formula $L_{\mathbf{1}_x}(T) = (1 - T)/(1 - pT)$ from~\eqref{eq:1x-gen}, Lemma~\ref{lem-fixconst:char-gen} implies the following corollary.

\begin{corollary}\label{cor-fixconst:char-gen2}
If $p$ is odd, then for any $\chi \in \widehat{\F_p^{\times}} \setminus \{1\}$ and $z \in \C$,
\[ \frac{L_{\chi \circ \ev_0}(T)}{L_{\chi^2 \circ \ev_0}(z T^2)} = \begin{cases} 1 - \dfrac{(p - 1) zT^2}{1 - zT^2} & \text{if } \chi = \left(\frac{\cdot}{p}\right), \\ 1 & \text{if } \chi \neq \left(\frac{\cdot}{p}\right). \end{cases} \]
\end{corollary}

We are now ready to prove Theorem~\ref{thm-fixconst:main}.
We first compute $P^{\max}(V_n(\Z_p)_{a_n = b_n})$ and $P_0^{\sqf}(V_n(\Z_p)_{a_n = b_n})$.
Using the formula~\eqref{eq-fixconst:weightFT} for $\hat{w}$, Theorem~\ref{thm-density-unit:max} and Theorem~\ref{thm-density-unit:sqf-0} respectively implies
\begin{align}
    P^{\max}(V_n(\Z_p)_{a_n = b_n}) &= \frac{1}{(p - 1) p^{n - 1}} \sum_{\chi \in \widehat{\F_p^{\times}}} \chi(\overline{b_n})^{-1} \left[\frac{L_{\chi \circ \ev_0}(T)}{L_{\chi^2 \circ \ev_0}(T^2/p)}\right]_n, \label{eq-fixconst:max-init} \\
    P_0^{\sqf}(V_n(\Z_p)_{a_n = b_n}) &= \frac{1}{(p - 1) p^{n - 1}} \sum_{\chi \in \widehat{\F_p^{\times}}} \chi(\overline{b_n})^{-1} \left[\frac{L_{\chi \circ \ev_0}(T)}{L_{\chi^2 \circ \ev_0}(T^2)}\right]_n. \label{eq-fixconst:sqf-0-init}
\end{align}
The first term of the right hand side of~\eqref{eq-fixconst:max-init} is equal to $P^{\max}(V_n(\Z_p)_{p \nmid a_n})$ by~\eqref{eq-copconst:max-init}, while the first term of the right hand side of~\eqref{eq-fixconst:sqf-0-init} is equal to $P_0^{\sqf}(V_n(\Z_p)_{p \nmid a_n})$ by~\eqref{eq-copconst:sqf-0-init}.
These are the only terms of the sum if $p = 2$, implying~\eqref{eq-fixconst:max-two} and~\eqref{eq-fixconst:sqf-0-two}.
For $p$ odd, the following proposition implies that~\eqref{eq-fixconst:max-odd-deg-odd} and~\eqref{eq-fixconst:max-odd-deg-even} follow from~\eqref{eq-fixconst:max-init}, while~\eqref{eq-fixconst:sqf-0-odd-deg-odd} and~\eqref{eq-fixconst:sqf-0-odd-deg-even} follow from~\eqref{eq-fixconst:sqf-0-init}.

\begin{proposition}
If $p$ is odd, then for any $n \geq 0$ and $z \in \C$,
\[ \frac{1}{(p - 1) p^{n - 1}} \sum_{\chi \in \widehat{\F_p^{\times}} \setminus \{1\}} \chi(\overline{b_n})^{-1} \left[\frac{L_{\chi \circ \ev_0}(T)}{L_{\chi^2 \circ \ev_0}(T^2/p)}\right]_n = -\frac{1}{p^{n - 1}} \left[\frac{zT^2}{1 - zT^2}\right]_n \left(\frac{b_n}{p}\right). \]
\end{proposition}
\begin{proof}
By Corollary~\ref{cor-fixconst:char-gen2}, the contribution from the quadratic character is
\[ \frac{1}{(p - 1) p^{n - 1}} \left(\frac{b_n}{p}\right)^{-1} \left[1 - \frac{(p - 1) zT^2}{1 - zT^2}\right]_n = -\frac{1}{p^{n - 1}} \left[\frac{zT^2}{1 - zT^2}\right]_n \left(\frac{b_n}{p}\right), \]
    while the contribution from the other characters is zero.
\end{proof}

Finally, we compute $P_1^{\sqf}(V_n(\Z_p)_{a_n = b_n})$.
Using the formula~\eqref{eq-fixconst:weightFT} for $\hat{w}$, Theorem~\ref{thm-density-unit:sqf-1} implies
\[ P_1^{\sqf}(V_n(\Z_p)_{a_n = b_n}) = \frac{1 - 1/p}{(p - 1) p^{n - 1}} \sum_{\chi \in \widehat{\F_p^{\times}}} \chi(\overline{b_n})^{-1} \left[\sum_{c \in \F_p^{\times}} \frac{\chi(\ev_0(x + c))^2}{1 + \chi(\ev_0(x + c)) T} \cdot \frac{L_{\chi \circ \ev_0}(T)}{L_{\chi^2 \circ \ev_0}(T^2)}\right]_{n - 2}. \]
By~\eqref{eq-copconst:sqf-1-init}, the term $\chi = 1$ contributes $P_1^{\sqf}(V_n(\Z_p)_{p \nmid a_n})$ to the right hand side.
It remains to compute the sum of the other terms.
By Corollary~\ref{cor-fixconst:char-gen2},
\begin{align*}
    & P_1^{\sqf}(V_n(\Z_p)_{a_n = b_n}) - P_1^{\sqf}(V_n(\Z_p)_{p \nmid a_n}) \\
    =\;& \frac{1 - 1/p}{(p - 1) p^{n - 1}} \sum_{\chi \in \widehat{\F_p^{\times}} \setminus \{1\}} \chi(\overline{b_n})^{-1} \left[\sum_{c \in \F_p^{\times}} \frac{\chi(\ev_0(x + c))^2}{1 + \chi(\ev_0(x + c)) T} \cdot \frac{L_{\chi \circ \ev_0}(T)}{L_{\chi^2 \circ \ev_0}(T^2)}\right]_{n - 2} \\
    =\;& \frac{1}{p^n} \sum_{\substack{\chi \in \widehat{\F_p^{\times}} \\ \chi^2 \neq 1}} \chi(\overline{b_n})^{-1} \left[\sum_{c \in \F_p^{\times}} \frac{\chi(c)^2}{1 + \chi(c) T}\right]_{n - 2} + \frac{1}{p^n} \left(\frac{b_n}{p}\right)^{-1} \left[\sum_{c \in \F_p^{\times}} \frac{\left(\frac{c}{p}\right)^2}{1 + \left(\frac{c}{p}\right) T} \left(1 - \frac{(p - 1) T^2}{1 - T^2}\right)\right]_{n - 2} \\
    =\;& \frac{1}{p^n} \sum_{\chi \in \widehat{\F_p^{\times}} \setminus \{1\}} \chi(\overline{b_n})^{-1} \left[\sum_{c \in \F_p^{\times}} \frac{\chi(c)^2}{1 + \chi(c) T}\right]_{n - 2} - \frac{1}{p^n} \left(\frac{b_n}{p}\right)^{-1} \left[\sum_{c \in \F_p^{\times}} \frac{\left(\frac{c}{p}\right)^2}{1 + \left(\frac{c}{p}\right) T} \cdot \frac{(p - 1) T^2}{1 - T^2}\right]_{n - 2} \\
    =\;& \frac{1}{p^n} \sum_{\chi \in \widehat{\F_p^{\times}} \setminus \{1\}} \chi(\overline{b_n})^{-1} \sum_{c \in \F_p^{\times}} (-1)^{n - 2} \chi(c)^n - \frac{1}{p^n} \left(\frac{b_n}{p}\right) \left[\frac{p - 1}{2} \left(\frac{1}{1 + T} + \frac{1}{1 - T}\right) \frac{(p - 1) T^2}{1 - T^2}\right]_{n - 2} \\
    =\;& \frac{(-1)^n}{p^n} (p - 1) \kappa_{n, p}(\overline{b_n}) - \frac{(p - 1)^2}{p^n} \left(\frac{b_n}{p}\right) \left[\frac{T^2}{(1 - T^2)^2}\right]_{n - 2},
\end{align*}
    where $\kappa_{n, p}$ is as defined in~\eqref{eq-fixconst:kappa}.
Finally, note that $\left[\dfrac{T^2}{(1 - T^2)^2}\right]_{n - 2}$ is equal to $0$ if $n$ is odd and $n/2 - 1$ if $n$ is even.
Substituting into the above formula yields~\eqref{eq-fixconst:sqf-1-odd-deg-odd} if $n$ is odd and~\eqref{eq-fixconst:sqf-1-odd-deg-even} if $n$ is even.
This finishes the proof of Theorem~\ref{thm-fixconst:main}.

\section{Proof of Theorem~\ref{thm-local:sub1-copconst}}\label{section:sub1-copconst}

Let $n \geq 2$ and $p$ be a prime number.
Fix $b_1 \in \Z_p$, and recall that
\begin{align*}
    V_n(\Z_p)_{a_1 = b_1, p \nmid a_n} &= \{f(x) = x^n + a_1 x^{n - 1} + \ldots + a_n \in \Z_p[x] : a_1 = b_1, p \nmid a_n\}, \\
    V_n(\Z_p)_{p \nmid a_n} &= \{f(x) = x^n + a_1 x^{n - 1} + \ldots + a_n \in \Z_p[x] : p \nmid a_n\}.
\end{align*}
We now state the exact formula for the densities of $V_n(\Z_p)_{a_1 = b_1, p \nmid a_n}$ in terms of the densities of $V_n(\Z_p)_{p \nmid a_n}$.
For any prime number $p$ and $b \in \F_p$, denote
\begin{equation}\label{eq-sub1-copconst:eta}
    \eta_p(b) = \sum_{\chi \in \widehat{\F_p} \setminus \{1\}} \chi(b)^{-1} = \begin{cases} p - 1 & \text{if } b = 0, \\ -1 & \text{if } b \neq 0. \end{cases}
\end{equation}
Then we prove the following result.

\begin{theorem}\label{thm-sub1-copconst:main}
If $p = 2$, then $P_1^{\sqf}(V_n(\Z_2)_{a_1 = b_1, 2 \nmid a_n}) = 0$ and
\begin{align}
    P^{\max}(V_n(\Z_2)_{a_1 = b_1, 2 \nmid a_n}) &= P^{\max}(V_n(\Z_2)_{2 \nmid a_n}) - \frac{(-1)^n}{2^{n + \lfloor n/2 \rfloor - 1}} \eta_2(\overline{b_1}), \label{eq-sub1-copconst:max-two} \\
    P_0^{\sqf}(V_n(\Z_2)_{a_1 = b_1, 2 \nmid a_n}) &= P_0^{\sqf}(V_n(\Z_2)_{2 \nmid a_n}) - \frac{(-1)^n}{2^{n - 1}} \eta_2(\overline{b_1}). \label{eq-sub1-copconst:sqf-0-two}
\end{align}

If $p$ is odd, then
\begin{align}
    P^{\max}(V_n(\Z_p)_{a_1 = b_1, p \nmid a_n}) &= P^{\max}(V_n(\Z_p)_{p \nmid a_n}) + \frac{(-1)^n}{(p - 1) p^{n + \lfloor n/2 \rfloor - 1}} \eta_p(\overline{b_1}), \label{eq-sub1-copconst:max-odd} \\
    P_0^{\sqf}(V_n(\Z_p)_{a_1 = b_1, p \nmid a_n}) &= P_0^{\sqf}(V_n(\Z_p)_{p \nmid a_n}) + \frac{(-1)^n}{(p - 1) p^{n - 1}} \eta_p(\overline{b_1}), \label{eq-sub1-copconst:sqf-0-odd} \\
    P_1^{\sqf}(V_n(\Z_p)_{a_1 = b_1, p \nmid a_n}) &= P_1^{\sqf}(V_n(\Z_p)_{p \nmid a_n}) - \frac{(-1)^n}{p^n} (n - 1 - p \lfloor n/p \rfloor) \; \eta_p(\overline{b_1}). \label{eq-sub1-copconst:sqf-1-odd}
\end{align}
\end{theorem}

Theorem~\ref{thm-local:sub1-copconst} follows from Theorem~\ref{thm-sub1-copconst:main}, Theorem~\ref{thm-local:copconst} on the asymptotic densities of $V_n(\Z_p)_{p \nmid a_n}$ and the estimate $\eta_p(b) = O(p)$.
To prove positivity of the densities, we just need to find a monic squarefree polynomial of degree $n$ over $\F_p$ whose constant coefficient is non-zero and whose $x^{n - 1}$-coefficient is $\overline{b_1}$.
If $p$ is odd, $x^n + \overline{b_1} x^{n - 1} + d$ works for some $d \in \F_p^{\times}$, as this polynomial has a double factor for only one possible choice of $d$.
The polynomial $x^n + \overline{b_1} x^{n - 1} + 1$ also works for $p = 2$ if either $n$ is odd or if $2 \nmid b_1$.
If $p = 2$, $n$ is even, and $2 \mid b_1$, then $x^n + x + 1$ works for $n \geq 4$.
In the exceptional case $(n, p) = 2$ with $2 \mid b_1$, all polynomials in the set $V_2(\Z_2)_{a_1 = b_1, 2 \nmid a_n}$ have discriminant divisible by $4$, but this set contains $f(x) = (x + b_1/2)^2 + 1$ and $\Z_2[x]/(f(x)) \cong \Z_2[i]$ is the maximal order of $\Q_2[x]/(f(x)) \cong \Q_2[i]$.

It remains to prove Theorem~\ref{thm-sub1-copconst:main}.
First, since squarefree discriminant and maximality are mod $p^2$ conditions, the densities for $V_n(\Z_p)_{a_1 = b_1, p \nmid a_n}$ are the same as those of the set
\[ V_n(\Z_p)_{a_1 \equiv b_1 \bmod{p^2}, p \nmid a_n} = \{f(x) = x^n + a_1 x^{n - 1} + \ldots + a_n \in \Z_p[x] : a_1 \equiv b_1 \bmod{p^2}, p \nmid a_n\}, \]
    which is defined by congruence conditions mod $p^2$.
We start by finding a $V_n(\Z_p)_{a_1 \equiv b_1 \bmod{p^2}, p \nmid a_n}$-unit-admissible triple $(G, \psi, w)$ as defined in Definition~\ref{def:unit-admissible} and computing $\hat{w}(\chi)$ and $L_{\chi \circ \psi}(T)$ for all $\chi \in \widehat{G}$.
Then we apply the formulas obtained in \S\ref{section:density-formula-unit} to prove Theorem~\ref{thm-sub1-copconst:main}.

Define the monoid homomorphism $\varphi_{1, x \nmid u} : \F_p[x]_{m, x \nmid u} \to \F_p$ by the formula
\begin{equation}\begin{split}\label{eq:varphi1-unit-def}
    \varphi_{1, x \nmid u}(x^n + a_1 x^{n - 1} + \ldots + a_n) &= a_1, \\
    \varphi_{1, x \nmid u}(1) &= 0.
\end{split}\end{equation}
Note that this map is just the restriction of the map $\varphi_1$ defined in~\eqref{eq:varphi1-def}.
It is easy to see that for any $u \in \F_p[x]_{m, x \nmid u}$ of degree $n$,
\[ \frac{\lambda_p(\{f \in V_n(\Z_p)_{a_1 \equiv b_1 \bmod{p^2}, p \nmid a_n} : \overline{f} = u\})}{\lambda_p(V_n(\Z_p)_{a_1 \equiv b_1 \bmod{p^2}, p \nmid a_n})} = \begin{cases} 1/((p - 1) p^{n - 2}) & \text{if } \varphi_{1, x \nmid u}(u) = \overline{b_1}, \\ 0 & \text{if } \varphi_{1, x \nmid u}(u) \neq \overline{b_1}. \end{cases} \]
Thus, $(\F_p, \varphi_{1, x \nmid u}, w)$ is $V_n(\Z_p)_{a_1 \equiv b_1 \bmod{p^2}, p \nmid a_n}$-unit-admissible, where the weight function $w : \F_p \to \C$ is defined by
\[ w(c) = \begin{cases} 1/((p - 1) p^{n - 2}) & \text{if } c = \overline{b_1}, \\ 0 & \text{if } c \neq \overline{b_1}. \end{cases} \]
By direct computation, the Fourier transform $\hat{w}$ of $w$ has the formula
\begin{equation}\label{eq-sub1-copconst:weightFT}
    \hat{w}(\chi) = \frac{1}{\# \F_p} \sum_{c \in \F_p} w(c) \chi(c)^{-1} = \frac{1}{(p - 1) p^{n - 1}} \chi(\overline{b_1})^{-1}.
\end{equation}

Now we compute $L_{\chi \circ \varphi_{1, x \nmid u}}(T)$ for any $\chi \in \widehat{\F_p}$ such that $\chi \neq 1$.
Note that $1 \circ \varphi_{1, x \nmid u} = \mathbf{1}_x$, where $\mathbf{1}_x : \F_p[x]_{m, x \nmid u} \to \C$ is the all-one function on $\F_p[x]_{m, x \nmid u}$.

\begin{lemma}\label{lem-sub1-copconst:char-gen}
For any $\chi \in \widehat{\F_p}$ such that $\chi \neq 1$, we have $L_{\chi \circ \varphi_{1, x \nmid u}}(T) = 1 - T$.
\end{lemma}
\begin{proof}
Clearly, $\left[L_{\chi \circ \varphi_{1, x \nmid u}}(T)\right]_0 = 1$ and since $\chi \neq 1$,
\[ \left[L_{\chi \circ \varphi_{1, x \nmid u}}(T)\right]_1 = \sum_{c \in \F_p^{\times}} \chi(c) = -1 + \sum_{c \in \F_p} \chi(c) = -1. \]
Now suppose that $n \geq 2$.
For each $c \in \F_p$, there exists exactly $(p - 1) p^{n - 2}$ polynomials in $\F_p[x]_{m, x \nmid u}$ whose $x^{n - 1}$-coefficient is equal to $c$.
As a result, we get
\[ \left[L_{\chi \circ \varphi_{1, x \nmid u}}(T)\right]_n = \sum_{\substack{u \in \F_p[x]_{m, x \nmid u} \\ \deg(u) = n}} \chi(\varphi_1(u)) = (p - 1) p^{n - 2} \sum_{c \in \F_p} \chi(c), \]
    which is equal to zero since $\chi \neq 1$.
\end{proof}

We are now ready to prove Theorem~\ref{thm-sub1-copconst:main}.
We first compute the densities $P^{\max}(V_n(\Z_p)_{a_1 = b_1, p \nmid a_n})$ and $P_0^{\sqf}(V_n(\Z_p)_{a_1 = b_1, p \nmid a_n})$.
Using the formula~\eqref{eq-sub1-copconst:weightFT} for $\hat{w}$, Theorem~\ref{thm-density-unit:max} and Theorem~\ref{thm-density-unit:sqf-0} respectively yields
\begin{align}
    P^{\max}(V_n(\Z_p)_{a_1 = b_1, p \nmid a_n}) &= \frac{1}{(p - 1) p^{n - 1}} \sum_{\chi \in \widehat{\F_p}} \chi(\overline{b_1})^{-1} \left[\frac{L_{\chi \circ \varphi_{1, x \nmid u}}(T)}{L_{\chi^2 \circ \varphi_{1, x \nmid u}}(T^2/p)}\right]_n, \label{eq-sub1-copconst:max-init} \\
    P_0^{\sqf}(V_n(\Z_p)_{a_1 = b_1, p \nmid a_n}) &= \frac{1}{(p - 1) p^{n - 1}} \sum_{\chi \in \widehat{\F_p}} \chi(\overline{b_1})^{-1} \left[\frac{L_{\chi \circ \varphi_{1, x \nmid u}}(T)}{L_{\chi^2 \circ \varphi_{1, x \nmid u}}(T^2)}\right]_n. \label{eq-sub1-copconst:sqf-0-init}
\end{align}
The contribution of the term $\chi = 1$ in~\ref{eq-sub1-copconst:max-init} is equal to $P^{\max}(V_n(\Z_p)_{p \nmid a_n})$ by~\eqref{eq-copconst:max-init}.
The contribution of the term $\chi = 1$ in~\ref{eq-sub1-copconst:sqf-0-init} is equal to $P_0^{\sqf}(V_n(\Z_p)_{p \nmid a_n})$ by~\eqref{eq-copconst:sqf-0-init}.
Thus the following proposition implies~\eqref{eq-sub1-copconst:max-two} and~\eqref{eq-sub1-copconst:sqf-0-two} for $p = 2$, and it implies~\eqref{eq-sub1-copconst:max-odd} and~\eqref{eq-sub1-copconst:sqf-0-odd} for $p$ odd.

\begin{proposition}
For any $n \geq 2$, prime number $p$, and $z \in \C$,
\[ \sum_{\chi \in \widehat{\F_p} \setminus \{1\}} \chi(\overline{b_1})^{-1} \left[\frac{L_{\chi \circ \varphi_{1, x \nmid u}}(T)}{L_{\chi^2 \circ \varphi_{1, x \nmid u}}(zT^2)}\right]_n = \begin{cases} -(-1)^n z^{\lfloor n/2 \rfloor} \eta_2(\overline{b_1}) & \text{if } p = 2, \\ (-1)^n z^{\lfloor n/2 \rfloor} \eta_p(\overline{b_1}) & \text{if } p > 2. \end{cases} \]
\end{proposition}
\begin{proof}
If $p = 2$, then $\widehat{\F_2} = \{1, \chi_2\}$, where $\chi_2$ is the character defined by $\chi_2(0) = 1$ and $\chi_2(1) = -1$.
It is easy to check that $\chi_2(b)^{-1} = \eta_2(b)$ for any $b \in \F_2$.
Therefore, by Lemma~\ref{lem-sub1-copconst:char-gen} and the formula $L_{\mathbf{1}_x}(T) = (1 - T)/(1 - pT)$ shown in~\eqref{eq:1x-gen}, we get
\[ \sum_{\chi \in \widehat{\F_2} \setminus \{1\}} \chi(\overline{b_1})^{-1} \left[\frac{L_{\chi \circ \varphi_{1, x \nmid u}}(T)}{L_{\chi^2 \circ \varphi_{1, x \nmid u}}(zT^2)}\right]_n = \left[\frac{(1 - T)(1 - 2zT^2)}{1 - zT^2}\right]_n \eta_2(\overline{b_1}). \]
The proposition follows since for any $n \geq 2$,
\[ \left[\frac{(1 - T)(1 - 2zT^2)}{1 - zT^2}\right]_n = \left[2(1 - T) - \frac{1 - T}{1 - zT^2}\right]_n = -\left[\frac{1 - T}{1 - zT^2}\right]_n = -(-1)^n z^{\lfloor n/2 \rfloor}. \]
If $p$ is odd, then $\chi \neq 1$ implies $\chi^2 \neq 1$, and so Lemma~\ref{lem-sub1-copconst:char-gen} implies
\[ \sum_{\chi \in \widehat{\F_p} \setminus \{1\}} \chi(\overline{b_1})^{-1} \left[\frac{L_{\chi \circ \varphi_{1, x \nmid u}}(T)}{L_{\chi^2 \circ \varphi_{1, x \nmid u}}(zT^2)}\right]_n = \sum_{\chi \in \widehat{\F_p} \setminus \{1\}} \chi(\overline{b_1})^{-1} \left[\frac{1 - T}{1 - zT^2}\right]_n = (-1)^n z^{\lfloor n/2 \rfloor} \eta_p(\overline{b_1}). \]
This proves the proposition.
\end{proof}

Finally, we compute $P_1^{\sqf}(V_n(\Z_p)_{a_1 = b_1, p \nmid a_n})$.
By Theorem~\ref{thm-density-unit:sqf-1} with the formula~\eqref{eq-sub1-copconst:weightFT} for $\hat{w}$, we get
\[ P_1^{\sqf}(V_n(\Z_p)_{a_1 = b_1, p \nmid a_n}) = \frac{1 - 1/p}{(p - 1) p^{n - 1}} \sum_{\chi \in \widehat{\F_p}} \chi(\overline{b_1})^{-1} \left[\sum_{c \in \F_p^{\times}} \frac{\chi(\varphi_{1, x \nmid u}(x + c))^2}{1 + \chi(\varphi_{1, x \nmid u}(x + c)) \, T} \cdot \frac{L_{\chi \circ \varphi_{1, x \nmid u}}(T)}{L_{\chi^2 \circ \varphi_{1, x \nmid u}}(T^2)}\right]_{n - 2}. \]
By~\eqref{eq-copconst:sqf-1-init}, the term $\chi = 1$ contributes $P_1^{\sqf}(V_n(\Z_p)_{p \nmid a_n})$ to the right hand side.
It remains to compute the sum of the other terms.
By Lemma~\ref{lem-sub1-copconst:char-gen},
\begin{align*}
    P_1^{\sqf}(V_n(\Z_p)_{a_1 = b_1, p \nmid a_n}) - P_1^{\sqf}(V_n(\Z_p)_{p \nmid a_n})
    &= \frac{1}{p^n} \sum_{\chi \in \widehat{\F_p} \setminus \{1\}} \chi(\overline{b_1})^{-1} \left[\sum_{c \in \F_p^{\times}} \frac{\chi(c)^2}{1 + \chi(c) \, T} \cdot \frac{1 - T}{1 - T^2}\right]_{n - 2} \\
    &= \frac{(-1)^{n - 2}}{p^n} \sum_{\chi \in \widehat{\F_p} \setminus \{1\}} \chi(\overline{b_1})^{-1} \left[\sum_{c \in \F_p^{\times}} \frac{\chi(c)^2}{1 - \chi(c) \, T} \cdot \frac{1 + T}{1 - T^2}\right]_{n - 2} \\
    &= \frac{(-1)^n}{p^n} \sum_{\chi \in \widehat{\F_p} \setminus \{1\}} \chi(\overline{b_1})^{-1} \left[\left(\frac{pT^{p - 2}}{1 - T^p} - \frac{1}{1 - T}\right) \frac{1}{1 - T}\right]_{n - 2} \\
    &= -\frac{(-1)^n}{p^n} \left(\left[\frac{1}{(1 - T)^2}\right]_{n - 2} - p \left[\frac{T^{p - 2}}{(1 - T^p)(1 - T)}\right]_{n - 2}\right) \eta_p(\overline{b_1}) \\
    &= -\frac{(-1)^n}{p^n} \left(n - 1 - p \sum_{k = 1}^{\infty} \left[\frac{T^{pk - 2}}{1 - T}\right]_{n - 2}\right) \eta_p(\overline{b_1}) \\
    &= -\frac{(-1)^n}{p^n} (n - 1 - p \lfloor n/p \rfloor) \; \eta_p(\overline{b_1}),
\end{align*}
    where the second line follows by replacing $T$ with $-T$, thus multiplying the $T^{n - 2}$-coefficient of the corresponding power series by $(-1)^{n - 2}$.
This finishes the proof of Theorem~\ref{thm-sub1-copconst:main}.

\subsection*{Acknowledgements}

The author would like to thank Jerry Wang for many helpful feedbacks on this paper.
The author is supported by an NSERC Discovery Grant.

\end{document}